\documentclass[11pt]{amsart}

\usepackage{amsmath,amssymb,amsthm}
\usepackage{amsfonts}
\usepackage{verbatim}
\usepackage[english]{babel}
\usepackage{longtable}
\usepackage{lscape}
\usepackage[latin1]{inputenc}
\usepackage{inputenc}
\usepackage{graphicx}
\usepackage{hyperref}

\usepackage[T1]{fontenc}
\usepackage{amscd}

\numberwithin{equation}{section}

\newtheorem{theo}{Theorem}[section]
\newtheorem{coro}[theo]{Corollary}
\newtheorem{lemma}[theo]{Lemma}
\newtheorem{propo}[theo]{Proposition}
\newtheorem{defn}[theo]{Definition}
\newtheorem{rk}[theo]{Remark}
\newtheorem{conj}{Conjecture}
\newcommand{\1}{1\hspace{-.55ex}\mbox{l}}

\newtheorem{thma}{Theorem \ref{limitset}}
\newtheorem{thmb}{Theorem \ref{PTA}}
\newtheorem{propa}{Proposition \ref{ultra}}
\newtheorem{propb}{Proposition \ref{compactness}}

\begin{document}
\author{Aline KURTZMANN}
\thanks{This work is partially supported by the Swiss National Foundation grants 200020-112316/1 and PBNE2-119027. The author is glad to thank M. Bena\"im for introducing her to the subject, P. Cattiaux for an introduction to the ultracontractivity property, V. Kleptsyn and P. Tarr\`es for helpful comments, and J.-F. Jouanin for valuable discussions.}
\title[Some self-interacting diffusions on $\mathbb{R}^d$]{The ODE method for some self-interacting diffusions on $\mathbb{R}^d$}

\maketitle

\begin{abstract}
The aim of this paper is to study the long-term behaviour of a class of self-interacting diffusion processes on $\mathbb{R}^d$. These are solutions to SDEs with a drift term depending on the actual position of the process and its normalized occupation measure $\mu_t$. These processes have so far been studied on compact spaces by Bena\"im, Ledoux and Raimond, using stochastic approximation methods% (establishing a relation between the asymptotic behavior of $\mu_t$ and the one from a deterministic dynamical flow)
. We extend these methods to $\mathbb{R}^d$, assuming a confinement potential satisfying some conditions. These hypotheses on the confinement potential are required since in general the process can be transient, and is thus very difficult to analyze. Finally, we illustrate our study with an example on $\mathbb{R}^2$.
\end{abstract}

\section{Introduction}
This paper addresses the long-term behavior of a class of
`self-interacting diffusion' processes $(X_t, t \geq 0)$ on
non-compact spaces. These processes are time-continuous,
non-Markov and live on $\mathbb{R}^d$. They are solutions to a kind of diffusion SDEs, whose drift term
depends on the whole past of the path through the occupation
measure of the process. Due to their non-Markovianity, they often exhibit an interesting ergodic behavior.

\subsection{Previous results on self-interacting diffusions}
Time-continuous self-interacting processes, also named `reinforced
processes', have already been studied in many contexts. Under the
name of `Brownian polymers', Durrett \& Rogers \cite{duR} first
introduced them as a possible mathematical model for the
evolution of a growing polymer. They are solutions of SDEs of the
form $$\mathrm{d}X_t = \mathrm{d}B_t + \mathrm{d}t\int_0^t \mathrm{d}s f(X_t-X_s)$$ where $(B_t; t\geq 0)$ is a standard Brownian motion on $\mathbb{R}^d$ and $f$ is a given function. As the process $(X_t; t\geq 0)$ evolves
in an environment changing with its past trajectory, this SDE
defines a self-interacting diffusion, %which can be 
either self-repelling or self-attracting, depending on $f$.
%In any dimension, Durrett \& Rogers obtained that $|X_t|/t$ is bounded (by a deterministic variable) whenever $f$ has a compact support. (For further results, we refer to \cite{crLJ,crM,heR,moT,rai}.)

%In the same spirit, 
Another modelisation, with dependence on the (convoled) normalized occupation measure $(\mu_t, t \geq 0)$, has also been considered since the work of Bena\"im, Ledoux \& Raimond \cite{beLR}. They introduced a process living in a compact
smooth connected Riemannian manifold $M$ without boundary:
\begin{equation}
\mathrm{d}X_t = \sum_{i=1}^N F_i(X_t)\circ\mathrm{d}B_t^i - \int_M \nabla_x W(X_t,y)\mu_t(\mathrm{d}y)\mathrm{d}t,
\end{equation}
where $W$ is a (smooth) interaction potential, $(B^1,\cdots,B^N)$
is a standard Brownian motion on $\mathbb{R}^N$ and the symbol
$\circ$ stands for the Stratonovich stochastic integration. The
family of smooth vector fields $(F_i)_{1\leq i\leq N}$ comes from
the H\"ormander `sum of squares' decomposition of the
Laplace-Beltrami operator $\Delta = \sum_{i=1}^N F_i^2.$ The
normalized occupation measure $\mu_t$ is defined by:
\begin{equation}\label{defmu}
\mu_t := \frac{r}{r+t}\mu +
\frac{1}{r+t}\int_0^t\delta_{X_s}\mathrm{d}s
\end{equation}
where $\mu$ is an initial probability measure and $r$ is a
positive weight. In the compact-space case, they showed that the
asymptotic behavior of $\mu_t$ can be related to the analysis of
some deterministic dynamical flow defined on the space of the
Borel probability measures. They went further in this study in
\cite{beR} and gave sufficient conditions for the a.s. convergence
of the normalized occupation measure. When the interaction is symmetric, $\mu_t$ converges a.s. to a local minimum
of a nonlinear free energy functional (each local minimum having a
positive probability to be chosen). All these results are
summarized in a recent survey of Pemantle \cite{pema}. 

The present paper follows the same lead and extends the
results of Bena\"im, Ledoux \& Raimond \cite{beLR} in the
non-compact setting. We present all results in the Euclidean space
$\mathbb{R}^d$ for the sake of simplicity, but, they can be extended to the case of
a complete connected Riemannian manifold $M$ without boundary with no
further difficulty than the use of notations and a bit of geometry. The point is to involve the Ricci curvature in the conditions and work on the space $M\backslash \text{cut}(o)$, where cut$(o)$ is the cut locus of $o$ (which has zero-mean).

\subsection{Statement of the problem}
Here we set the main definitions: let us consider a confinement
potential $V: \mathbb{R}^d \rightarrow \mathbb{R}_{+}$ and an interaction potential $W: \mathbb{R}^d \times \mathbb{R}^d \rightarrow \mathbb{R}_{+}$. For any bounded Borel measure $\mu$,
we consider the `convoled' function
\begin{equation}
W*\mu: \mathbb{R}^d \rightarrow \mathbb{R}, \, W*\mu(x) :=
\int_{\mathbb{R}^d} W(x,y)\mu(\mathrm{d}y).
\end{equation}
Our main object of interest is the self-interacting diffusion
solution to
\begin{eqnarray}\label{eds}
\left\{%
\begin{array}{ll}
     &\mathrm{d}X_t = \mathrm{d}B_t -\left( \nabla V(X_t)+
\nabla W*{\mu_t}(X_t) \right) \mathrm{d}t \\
     &\mathrm{d}\mu_t = (\delta_{X_t} - \mu_t)\frac{\mathrm{d}t}{r+t}\\
     &X_0 = x,\, \mu_0=\mu \\
\end{array}
\right. \label{diffu}
\end{eqnarray}
where $(B_t)$ is a $d$-dimensional Brownian motion. Our goal is to study the long-term behavior of $(\mu_t, t\geq 0)$. Let us recall that the main difference with the previous work \cite{beLR} is that the state space is $\mathbb{R}^d$ and hence is not compact anymore. However, we will be able to extend the results obtain in the compact case: the behavior of $\mu_t$ is closely related to the behavior of a deterministic flow. We will also give some sufficient conditions on the interaction potential in order to prove ergodic results for the process $(X_t,t\geq 0)$.

Before stating the theorems proved in this paper, let us briefly describe the main results obtained so far on self-interacting diffusions in non-compact spaces. They concern the model of Durrett \& Rogers, and can be classified in three categories. The first one is when $f$ is real, compactly supported and its sign is constant. Cranston \& Mountford \cite{crM} have solved a (partially proved) conjecture of Durrett \& Rogers and shown that $X_t/t$ converges a.s. The second one deals with attracting interaction on $\mathbb{R}$ (i.e. $xf(x)\le 0$ for all $x\in\mathbb{R}$) 
studied in the constant case by Cranston \& Le Jan \cite{crLJ} and its generalization by Raimond \cite{rai} in the $d$-dimensional ``constant" case $f(x)=-ax/\|x\|$, or by Herrmann \& Roynette \cite{heR} for a local interaction. Under some conditions, it is proved that $X_t$ converges a.s., whereas for a non-local interaction, it does not in general (but the paths are a.s. bounded for $f(x) = -\text{sign} (x) \1_{|x|\geq a}$). The third one concerns a non-integrable repulsive $f$ on $\mathbb{R}$ (i.e. $xf(x)\ge 0$ for all $x\in\mathbb{R}$) studied by Mountford \& Tarr\`es \cite{moT} and solving a conjecture of Durrett \& Rogers. They have proved that for $f(x) = x/(1+|x|^{1+\beta})$, with $0<\beta<1$, there exists a positive $c$ such that with probability $1/2$, the symmetric process $t^{-2/(1+\beta)}X_t$ converges to $c$.

These previous works have in common that the drift may overcome the noise, so that the randomness of the process is ``controlled". To illustrate that, let us mention, for the same model of Durrett \& Rogers, the case of a repulsive function $f$ of compact support, also conjectured in \cite{duR}, which is still unsolved.
\begin{conj}\cite{duR}
Suppose that $f: \mathbb{R}\rightarrow \mathbb{R}$ is an odd function, of compact support. Then $X_t /t$ converges a.s. to 0.
\end{conj}
Coming back to our process of interest, the role of the confinement potential is to similarly ``control" the drift term of the diffusion. Indeed, for the process \eqref{eds} with $V=0$, the interaction potential is in general not strong enough for the process to be recurrent, and the behavior is then very difficult to analyze. In particular, it is hard to predict the relative importance of the drift term (in competition against the Brownian motion) in the evolution.

\subsubsection{Technical assumptions on the potentials} In the sequel, $(\cdot,\cdot)$ stands for the Euclidian scalar product. We denote by \textbf{(H)} the following hypotheses:
\begin{itemize}
%\left\{
    \item[i)] (\textit{regularity and positivity}) $V \in \mathcal{C}^2(\mathbb{R}^d)$ and $W\in
    \mathcal{C}^2(\mathbb{R}^d\times \mathbb{R}^d)$ and $V \geq 1, \; W \geq 0$;
    \item[ii)] (\textit{convexity}) $V$ is a strictly uniformly convex function: there exists $K>0$ such that for all $x, \xi \in \mathbb{R}^d$: $(\nabla^2 V(x) \xi,\xi) \geq K|\xi|^2$;
    \item[iii)] (\textit{growth}) there exist $c, C>0$, $\delta >1$ such that for all $x$ large enough, $(\nabla V (x),x) \geq c|x|^{2\delta}$ and for all $x,y\in \mathbb{R}^d$ 
    \begin{equation}\label{growth} \vert \nabla V(x) -\nabla V(y) \vert \leq C (\vert x -y\vert \wedge 1) (V(x) + V(y));
    \end{equation}
    \item[iv)] (\textit{domination}) there exists $\kappa\geq 1$ such that for all $x,y \in \mathbb{R}^d$,
    \begin{equation}\label{domination}
    W(x,y) + |\nabla_x W(x,y)| + |\nabla^2_{xx}W(x,y)|\leq \kappa \left(V(x) + V(y) \right);
    \end{equation}
    \item[v)] (\textit{curvature}) there exist $\alpha>-1$, $M\in \mathbb{R}$ such that for all $x,y,\xi \in \mathbb{R}^d$ 
\begin{equation}\label{curvature}
\frac{(x,\nabla_x W(x,y))}{(x,\nabla V(x))} \rightarrow \alpha \, \, \text{and} \, \, \left((\nabla^2 V(x) + \nabla^2_{xx} W(x,y))\xi,\xi\right) \geq M \vert \xi\vert^2.
\end{equation}
%\right.
\end{itemize}

\begin{rk} 1) The most important conditions are the domination iv) and the curvature v).\\
\noindent 2) The growth condition (\ref{growth}) on $V$ ensures that there exists $a>0$ such that for all $x\in \mathbb{R}^d$, we have
\begin{equation}\label{deltaV}
\Delta V(x) \leq aV(x).
\end{equation}
\noindent 3) The positivity and domination conditions
(\ref{domination}) on the interaction potential are not so hard to
be satisfied, since the self-interacting process will be invariant
by the gauge transform $W(x,y)\mapsto W(x,y)+\phi(y)$ for any
function $\phi$ that does not grow faster than $V$.
\end{rk}
\subsubsection{Results}
We can now describe the behavior of $\mu_t$. %Let $\Phi$ be the dynamical system defined, on a subset of the probability measures of $\mathbb{R}^d$, by: $$\Phi_t(\mu) = e^{-t} \mu + e^{-t}\int_0^t e^s \Pi(\Phi_s(\mu)) \mathrm{d}s, \, \, \Phi_0(\mu) = \mu.$$
\begin{theo}\label{th:intro}
Suppose \textbf{(H)}.\\ 
1) $\mathbb{P}_{x,r,\mu}$-a.s., the $\omega$-limit set of $(\mu_t,t\geq 0)$ is weakly compact, invariant by $\Phi$ and admits no other (sub-)attractor than itself.\\
2) If $W$ is symmetric, then, $\mathbb{P}_{x,r,\mu}$-a.s., the $\omega$-limit set of $(\mu_t,t\geq 0)$ is a connected subset of set of fixed points of the probability measure proportional to $e^{-2(V+W*\mu)(x)} \mathrm{d}x$.
\end{theo}

Even if the model studied could at a first glance seem restrictive (because of $V$), the drift term can really compete against the Brownian motion, as shown by the following:
\begin{theo}\label{th:ex-intro}
Consider the self-interacting diffusion on $\mathbb{R}^2$, with $V(x) = V(|x|)$ and $W(x,y) =(x,Ry)$. Let $\gamma(\rho) := e^{-2V(\rho)} /Z$. Then one of the following holds:
\begin{enumerate}
    \item If $V$ is such that $\int_0^\infty \mathrm{d}\rho \gamma(\rho) \rho^2 \cos (\theta) > -1$,
    then a.s. $\mu_t \xrightarrow {(w)} \gamma$;
    \item Else, we get two different cases:

    a) if $\theta = \pi$ then there exists a random measure $\mu_\infty$ such that a.s. $\mu_t\xrightarrow {(w)} \mu_\infty$,\\

    b) if $\theta \neq \pi$, then the $\omega$-limit set $\omega(\mu_t, t\geq 0) = \{\nu(\delta), 0\leq \delta<2\pi\}$ a.s., that is $\mu_t$ circles around.
\end{enumerate}
\end{theo}

\subsection{Outline of contents}
As mentioned earlier, the main difficulty here stems from the non-compactness of the state space. The way to get around it is first to introduce the $V$-norm in Section 2 (also named `dual weighted norm') and then show that $(\mu_t,t\geq 0)$ is a tight family of measures in \S 6.1. Second, the dynamical system involved in the study induces only a local semiflow
and not necessarily a global one. But, we will show in \S 4.2 (for some cases), that the semiflow does not explode. Last, define the Feller diffusion $X^\mu$ obtained by fixing the occupation measure $\mu_t$ (appearing in the drift term) to $\mu$. Let note $A_\mu$ its infinitesimal generator and $Q_\mu$ its fundamental kernel, that is $A_\mu \circ Q_\mu = \Pi(\mu) -Id$, where $\Pi(\mu)$ is the invariant probability measure of $X^\mu$. An essential point of our study consists in finding an upper bound for the operator $Q_\mu$. Indeed, one has to use here the notion of (uniform) ultracontractivity in \S 5.1.

The organization of this paper is as follows. In the next section, we introduce some notations and the deterministic flow involved. We will also prove the existence and uniqueness of the random process studied. Section 3 is devoted to the presentation of the main results and is divided in three parts. First, we recall the former results and ideas of Bena\"im and \textit{al} \cite{beLR}. Then, we state the tightness of $(\mu_t)_t$ and introduce the uniform estimates on the Feller semigroup. We end this section by describing the behavior of $\mu_t$. After that, we analyze, in Section 4, the deterministic semiflow associated to the self-interacting diffusion. We will prove the local existence of the semiflow and introduce two important objects: the convex hull of $\text{Im}(\Pi)$ and the fixed points of $\Pi$. Then,
in Section 5, we study in details the family of Markov semigroups, corresponding to $X^\mu$, for which we prove the uniform ultracontractivity property and the regularity of the operators $A_\mu$ and $Q_\mu$. The proofs of the main results are in Section 6, which heavily relies on the spectral analysis of the preceding section. We begin Section 6 by showing the tightness of $(\mu_t)_t$ in \S 6.1. Then, \S 6.2 deals with the approximation of the normalized occupation measure $(\mu_t, t\geq 0)$ by a deterministic semiflow. In \S 6.3, we prove Theorem \ref{th:intro}. Finally, Section 7 is devoted to the illustration in dimension $d=2$ stated in Theorem \ref{th:ex-intro}.

\section{Preliminaries and Tools}
\subsection{Some useful measure spaces}
We denote by $\mathcal{M}(\mathbb{R}^d)$ the space of
signed (bounded) Borel measures on $\mathbb{R}^d$ and by
$\mathcal{P}(\mathbb{R}^d)$ its subspace of probability measures.
We will need the following measure space:
\begin{equation}
\mathcal{M}(\mathbb{R}^d;V):= \{\mu \in \mathcal{M}(\mathbb{R}^d);
\int_{\mathbb{R}^d} V(y) \vert\mu\vert(\mathrm{d}y)<\infty\},
\end{equation}
where $\vert\mu\vert$ is the variation of $\mu$ (that is
$\vert\mu\vert := \mu^{+} + \mu^{-}$ with $(\mu^{+},\mu^{-})$ the
Hahn-Jordan decomposition of $\mu$). This space will enable us to
always check the integrability of $V$ (and therefore of $W$ and
its derivatives thanks to the domination condition
(\ref{domination})) with respect to the measures to be considered.
For example it contains the measure
\begin{equation}
\gamma(\mathrm{d}x):=\exp{\left(-2V(x)\right)}\mathrm{d}x.
\end{equation}
We endow $\mathcal{M}(\mathbb{R}^d;V)$ with the following dual
weighted supremum norm (or dual $V$-norm) defined by
\begin{equation}
\vert \vert \mu\vert \vert_V:= \sup_{\varphi; \vert \varphi\vert
\leq V} \left\vert \int_{\mathbb{R}^d} \varphi
\mathrm{d}\mu\right\vert, \quad \mu \in
\mathcal{M}(\mathbb{R}^d;V).
\end{equation}
This norm naturally arises in the approach to ergodic results for
time-continuous Markov processes of Meyn \& Tweedie \cite{meT}. It makes $\mathcal{M}(\mathbb{R}^d;V)$ a Banach space. To
illustrate the need of this space, we state an easy result that
will be used many times:
\begin{lemma}\label{majW}
For any $\mu \in \mathcal{M}(\mathbb{R}^d;V)$ the function $W *
\mu$ belongs to $\mathcal{C}^2(\mathbb{R}^d)$ and we have the
estimate
\begin{equation*}
\vert W * \mu(x) \vert \leq 2\kappa \vert \vert \mu \vert \vert_V
V(x).
\end{equation*}
\end{lemma}
\begin{proof}
Straightforward thanks to the domination condition
(\ref{domination}).
\end{proof}
In the following, we will consider probability measures only. Thus
we set, $\mathcal{P}(\mathbb{R}^d;V) :=
\mathcal{M}(\mathbb{R}^d;V) \cap \mathcal{P}(\mathbb{R}^d)$. The
strong topology on $\mathcal{P}(\mathbb{R}^d;V)$ is the trace
topology of the one defined on $\mathcal{M}(\mathbb{R}^d;V)$. It
makes $\mathcal{P}(\mathbb{R}^d;V)$ a complete metric space (for
the norm distance). Finally, for any $\beta>1$, we introduce the
subspace
\begin{equation}
\mathcal{P}_\beta(\mathbb{R}^d;V) :=\{\mu \in
\mathcal{P}(\mathbb{R}^d); \int_{\mathbb{R}^d} V(y) \mu
(\mathrm{d}y)\leq \beta\}.
\end{equation}

\subsection{The family of semigroups $(P_t^{\mu})$}
In all the following, $(\Omega, \mathcal{F},(\mathcal{F}_t)_t,\mathbb{P})$ will be a filtered probability
space satisfying the usual conditions. For any Borel probability
measure $\mu \in \mathcal{P}(\mathbb{R}^d;V)$, let $(X_t^\mu,t\geq
0)$ be the Feller diffusion defined by the following SDE
\begin{eqnarray}
\left\{%
\begin{array}{ll}
    & \mathrm{d}X_t^\mu = \mathrm{d}B_t -\left( \nabla V(X_t^\mu)+
\nabla W*{\mu}(X_t^\mu) \right) \mathrm{d}t, \\
    & X_0^\mu = x. \\
\end{array}
\right. \label{diffubis}
\end{eqnarray}
\begin{propo}\label{non-expl}
The diffusion $X^\mu_t$ a.s. never explodes.
\end{propo}
\begin{proof}
It is enough to check with the It\^o formula that the function
\begin{equation}\label{lyapunov}
\mathcal{E}_{\mu}(x):=V(x)+W*\mu(x).
\end{equation}
is a Lyapunov function. To see it, we notice that the growth and
domination conditions \eqref{growth} and \eqref{domination} on the
potentials imply the existence of $D>0$ such that:
\begin{equation}\label{deltaE}
\Delta \mathcal{E}_{\mu}(x) \leq 2DV(x) \leq
2D\mathcal{E}_{\mu}(x).
\end{equation}
As a by-product we get the naive (but useful!) estimate
\begin{equation}\label{Xmu_bound}
\mathbb{E}\mathcal{E}_\mu(X^\mu_t)\leq \mathcal{E}_\mu(x)e^{Dt}.
\end{equation}
\end{proof}
Now we denote by $(P^\mu_t;t\geq 0)$ the Markov semigroup
associated to $X^\mu_t$. We consider the differential operator
$A_\mu$ defined on $\mathcal{C}^\infty (\mathbb{R}^d)$ by
\begin{eqnarray}
A_\mu f &:=& \frac{1}{2}\Delta f - (\nabla V +\nabla W*\mu, \nabla f) 
\end{eqnarray}
$A_\mu$ corresponds to the infinitesimal generator of the true
diffusion (\ref{diffubis}). We emphasize that $(X_t^\mu)$ is a
positive-recurrent (reversible) diffusion and denote by
$\Pi(\mu)\in\mathcal{P}(\mathbb{R}^d;V)$ its unique invariant
probability measure:
\begin{equation}
\Pi(\mu)(\mathrm{d}x) := \frac{e^{-2W*\mu (x)}}{Z(\mu)} \gamma
(\mathrm{d}x)
\end{equation}
where $Z(\mu) := \int_{\mathbb{R}^d} e^{-2W*\mu (x)}
\gamma(\mathrm{d}x) < +\infty$ is the normalization constant. To
end this part, we recall the classical ergodic theorem for
$X^\mu_t$. We introduce the weighted supremum norm (or $V$-norm)
\begin{equation}
\vert\vert f\vert\vert_V := \sup_{x\in \mathbb{R}^d} \frac{\vert f(x)\vert}{V(x)},
\end{equation}
and the space of continuous $V$-bounded functions
\begin{equation}
\mathcal{C}^0(\mathbb{R}^d;V) := \{f \in
\mathcal{C}^0(\mathbb{R}^d) : \vert\vert f\vert\vert_V < \infty \}
\end{equation}
Similarly let $\mathcal{C}^p(\mathbb{R}^d;V) :=
\mathcal{C}^p(\mathbb{R}^d) \cap \mathcal{C}^0(\mathbb{R}^d;V)$
for all integer $p\geq 1$. The ergodic theorem says that a.s. we
have, for all $f\in\mathcal{C}^0(\mathbb{R}^d;V)$:
\begin{equation}\label{ergodic0}
\lim_{t\rightarrow \infty}\frac{1}{1+t}\int_0^t
f(X^\mu_s)\mathrm{d}s = \Pi(\mu)f.
\end{equation}

\subsection{The infinite-dimensional ODE}
We introduce a dynamical system on the set of probability measures
$\mathcal{P}(\mathbb{R}^d;V)$. We will assume the existence of the
semiflow $\Phi: \mathbb{R}_{+}\times \mathcal{P}(\mathbb{R}^d;V)
\rightarrow \mathcal{P}(\mathbb{R}^d;V)$ defined by
\begin{eqnarray}\label{flow}
\left\{%
\begin{array}{ll}
    & \Phi_t(\mu) = e^{-t} \mu + e^{-t}\int_0^t e^s \Pi(\Phi_s(\mu))\,
\mathrm{d}s, \\
    & \Phi_0(\mu) = \mu.\\
\end{array}
\right.
\end{eqnarray}
\begin{rk}
In Section 4, we will prove the local existence
of the semiflow, and, for $W$ symmetric or bounded, we will prove
it never explodes.
\end{rk}

In order to study the semiflow $\Phi$, we will need to endow the
space $\mathcal{P}(\mathbb{R}^d;V)$ with different topologies.
When nothing else is stated, we will consider that it is endowed
with the strong topology defined by the dual weighted supremum
norm $\vert \vert \cdot \vert \vert_V$. But, as the reader will
notice, we will frequently need to switch from the strong topology
to the weak topology of convergence of measures. We adopt here a
non-standard definition compatible with possibly unbounded
functions (yet dominated by $V$): for any sequence of probability
measures $(\mu_n,n\geq 1)$ and any probability measure $\mu$ (all
belonging to $\mathcal{P}(\mathbb{R}^d;V)$), we define the weak
convergence as:
\begin{equation}
\mu_n \xrightarrow {w} \mu \,\, \text{if and only if } \, \int_{\mathbb{R}^d} \varphi \, \mathrm{d}\mu_n \underset{n\rightarrow\infty}{\longrightarrow}\int_{\mathbb{R}^d}
\varphi \, \mathrm{d}\mu, \, \forall \varphi \in \mathcal{C}^0(\mathbb{R}^d;V).
\end{equation}
We point out that our definition of the weak convergence is
stronger than the usual one. We recall that
$\mathcal{P}(\mathbb{R}^d;V)$, equipped with the weak topology, is
a metrizable space. Since $\mathcal{C}^0(\mathbb{R}^d;V)$ is
separable, we exhibit a sequence $(f_k)_k$ dense in $\{ f\in
\mathcal{C}^0(\mathbb{R}^d;V) / ||f||_V\leq 1\}$, and set for all
$\mu, \nu \in \mathcal{P}(\mathbb{R}^d;V)$:
\begin{eqnarray}\label{distance}
\mathrm{d}(\mu, \nu) := \sum_{k=1}^\infty 2^{-k} \vert \mu (f_k) -
\nu (f_k)\vert.
\end{eqnarray}
Then the weak topology is the metric topology generated by
$\mathrm{d}$.

\subsection{The self-interacting diffusion}
We recall the self-interacting diffusion considered here:
\begin{eqnarray*}
\left\{%
\begin{array}{ll}
    & \mathrm{d}X_t = \mathrm{d}B_t -\left( \nabla V(X_t)+
\nabla W*{\mu_t}(X_t) \right) \mathrm{d}t \\
     & \mathrm{d}\mu_t = (\delta_{X_t}-\mu_t)\frac{\mathrm{d}t}{r+t}\\
    & X_0 = x, \, \mu_0=\mu\\
\end{array}
\right.
\end{eqnarray*}

\begin{propo}\label{existence} For any $x\in \mathbb{R}^d$,
$\mu\in\mathcal{P}(\mathbb{R}^d;V)$ and $r>0$, there exists a
unique global strong solution $(X_t,\mu_t, t\geq 0)$.
\end{propo}
\begin{proof}
First, we point out that, for all $t>0$ such that $(X_s,s\leq t)$
is defined, $\mu_t \in \mathcal{P}(\mathbb{R}^d;V)$. In order to
show that the solution never explodes, we use again the Lyapunov
functional $(x,\mu)\mapsto \mathcal{E}_\mu(x)$ (see \eqref{lyapunov}). As the process $(t,x)\mapsto\mathcal{E}_{\mu_t}(x)$ is of class $\mathcal{C}^2$ (in the space variable) and is a $\mathcal{C}^1$-semi-martingale
(in the time variable), the generalized It\^o formula (or
It\^o-Ventzell formula, see \cite{kun}), applied to
$(t,x)\mapsto\mathcal{E}_{\mu_t}(x)$ implies
\begin{eqnarray*}
\mathcal{E}_{\mu_t}(X_t) &=& \mathcal{E}_{\mu}(x) + \int_0^t
(\nabla \mathcal{E}_{\mu_s}(X_s),
\mathrm{d}B_s) - \int_0^t \left|\nabla \mathcal{E}_{\mu_s}(X_s)\right|^2 \mathrm{d}s\\
&+& \frac{1}{2}\int_0^t \Delta \mathcal{E}_{\mu_s}(X_s)\mathrm{d}s
+ \int_0^t \left(W(X_s,X_s) - W*\mu_s(X_s)
\right)\frac{\mathrm{d}s}{r+s}.
\end{eqnarray*}
Let us introduce the sequence of stopping times $$\tau_n :=
\inf\{t\geq 0; \mathcal{E}_{\mu_t}(X_t) + \int_0^t \left|\nabla
\mathcal{E}_{\mu_s}(X_s)\right|^2 \mathrm{d}s > n\}.$$ We note
that $\int_0^{t\wedge \tau_n} (\nabla \mathcal{E}_{\mu_s}(X_s),
\mathrm{d}B_s)$ is a true martingale. Again equation
\eqref{lyapunov} implies
\begin{equation*}
\mathbb{E}\mathcal{E}_{\mu_{t\wedge \tau_n}}(X_{t\wedge \tau_n})
\leq \mathcal{E}_{\mu}(x) + D \int_0^{t}
\mathbb{E}\mathcal{E}_{\mu_{s\wedge \tau_n}}(X_{s\wedge \tau_n})
\mathrm{d}s.
\end{equation*}
Therefore Gronwall's lemma leads to the same kind of estimate as
for $X^\mu$: $$\mathbb{E}\mathcal{E}_{\mu_{t\wedge
\tau_n}}(X_{t\wedge \tau_n}) \leq \mathcal{E}_\mu(x)e^{Dt}.$$
As $\underset{|x|\rightarrow \infty}{\lim}V(x) =\infty$, the process $(X_t,t\geq 0)$ does not explode in a finite time and the SDE (\ref{diffu}) admits a global strong solution.
\end{proof}

\section{Main results}

\subsection{Former tools and general idea}
We recall how Bena\"im, Ledoux \& Raimond handled in \cite{beLR} the asymptotic behavior of $\mu_t$ in a compact space. 
Indeed, we sketch here the general idea of the study and explain why the
tools introduced in the preliminary Section arise quite
naturally.

Begin to consider that the occupation measure appearing in the
drift is `frozen' to some fixed measure $\mu$. We obtain the
Feller diffusion $X^\mu_t$. For this diffusion, it is easy to
prove the existence of a spectral gap, and that the semigroup
$(P^\mu_t;t \geq 0)$ is exponentially $V$-uniformly ergodic:
\begin{equation}\label{ergodic1}
\vert \vert P^\mu_t f - \Pi(\mu)f \vert \vert_V \leq K(\mu)\vert
\vert f \vert \vert_V e^{-c(\mu)t}, \quad f \in
\mathcal{C}^0(\mathbb{R}^d;V).
\end{equation}
To get, as by-product, the almost sure convergence of the
empirical occupation measure of the process $X^\mu_t$ (see
\eqref{ergodic0}), a standard technique is to consider the
operator (sometimes called the `fundamental kernel' as in
Kontoyiannis \& Meyn \cite{koM}) for any $f \in
\mathcal{C}^\infty(\mathbb{R}^d;V)$
\begin{equation}
Q_\mu f := \int_0^\infty \left(P_t^\mu f - \Pi(\mu)f\right)
\mathrm{d}t
\end{equation}
Then it is enough to apply the It\^o formula to $Q_\mu f(X^\mu_t)$ and divide
both members by $t$ to get the desired result. Indeed one has
\begin{equation*}
Q_\mu f(X^\mu_t) = Q_\mu f(x) + \int_0^t \left( \nabla Q_\mu
f(X^\mu_s), \mathrm{d}B_s \right) + \int_0^t A_\mu Q_\mu
f(X^\mu_s) \mathrm{d}s.
\end{equation*}
Some easy bounds on the semigroup $(P^\mu_t)$ are enough to prove
that almost all terms are negligible compared to $t$ and it
remains to recognize the third term since $A_\mu Q_\mu f =
\Pi(\mu)f - f$.\medskip

Now when $\mu_t$ changes in time, we still can write a convenient
extended form of the It\^o formula (which let appear the time
derivative of $Q_{\mu_t} f(x)$) but we need to improve the
remainder of the argument. Intuitively, the distance between the time-derivative of $\mu_{e^t}$ and the term $\Pi(\mu_{e^t}) -\mu_{e^t}$ converges to zero a.s. As for  stochastic approximation processes, one expects the trajectories of $\mu_t$ to approximate the trajectories of a deterministic semiflow, meaning that the empirical measure $\mu_t$ is an asymptotic pseudotrajectory for the semiflow $\Phi$ induced by $\Pi(\mu) - \mu$. This very last remark conveyed to Bena\"im \& \textit{al} \cite{beLR} the idea of comparing the asymptotic evolution of $(\mu_t; t\geq 0)$ with the semiflow $(\Phi_t(\mu))$.\medskip

The notion of asymptotic pseudotrajectory was first introduced in
Bena\"im \& Hirsch \cite{beH}. It is particularly useful to analyze the long-term behavior of stochastic processes,
considered as approximations of solutions of ordinary differential
equation (the ``ODE method"). Let us give here some definitions.

\begin{defn} i) For every continuous function $\xi: \mathbb{R}_+\rightarrow \mathcal{P}(\mathbb{R}^d;V)$, the $\omega$-limit set of $\xi$, denoted by $\omega(\xi_t,t\geq 0)$, is the set of limits of weak convergent sequences $\xi(t_k), t_k \uparrow \infty$, that is \begin{equation} \omega(\xi_t,t\geq 0) := \bigcap_{t\geq 0} \overline{\xi([t,\infty))}, \end{equation} where $\overline{\xi([t,\infty))}$ stands for the closure of $\xi([t,\infty))$ according to the weak topology. 

ii) A continuous function $\xi: \mathbb{R}_+ \rightarrow
\mathcal{P}(\mathbb{R}^d;V)$ is an \textrm{asymptotic
pseudotrajectory} (or \textrm{asymptotic pseudo-orbit}) for the
semiflow $\Phi$ if for all $T>0$,
\begin{equation}
\lim_{t\rightarrow +\infty}\, \sup_{0\leq s\leq T}\, \mathrm{d}(\xi_{t+s}, \Phi_s(\xi_t) ) = 0.
\end{equation}
\end{defn}
The purpose here is to find an asymptotic pseudotrajectory for the
semiflow $\Phi$ defined by \eqref{flow}.

\subsection{New tools: tightness and uniform estimates}
We will prove that we can find $\beta>1$ such that $\mu_t \in \mathcal{P}_\beta(\mathbb{R}^d;V)$ $\mathbb{P}_{x,r,\mu}-$a.s. for $t$ large enough. Remark, that this last set is compact for the weak topology, and so $(\mu_t,t\geq 0)$ is a.s. tight. Then, we have to obtain precise bounds on the family of semigroups $(P_t^\mu,t\geq 0)$ where $\mu\in \mathcal{P}_\beta(\mathbb{R}^d;V)$. Obviously, while using the ergodic estimates (\ref{ergodic1}), we would like to get bounds that are uniform in $\mu$.

\subsubsection{Tightness}
The paper of Bena\"im \& \textit{al} \cite{beLR} crucially relies on the compactness of the manifold $M$ where the diffusion lives. This compactness readily implies that the process $(\mu_t,t\geq 0)$ will very fast be close to the `invariant' probability measure $\Pi(\mu_t)$. On the contrary, if the state space is $\mathbb{R}^d$ and $V\equiv 0$, then $X$ will escape from any compact set. Indeed, the confinement potential $V$ forces the process $(\mu_t, t \geq 0)$ to remain in a (weakly) compact space of measures, for $t$ large, and $X$ is then recurrent. We first exhibit a useful weakly compact set, for $\beta>0$:
\begin{propo} \label{P_beta}
$\mathcal{P}_\beta(\mathbb{R}^d;V)$ is a weakly compact subset of $\mathcal{P}(\mathbb{R}^d;V)$.
\end{propo}
\begin{proof}
It is clear that $\mathcal{P}_\beta(\mathbb{R}^d;V)$ is weakly
closed. The Prohorov theorem shows that it is enough to prove that
$\mathcal{P}_\beta(\mathbb{R}^d;V)$ is tight. For every $a>0$ and
$\mu \in \mathcal{P}_\beta(\mathbb{R}^d;V)$, we have
\begin{equation*}
\underset{|x|>a}{\mathrm{min}} V(x)\; \mu\left(\{ |x|>a \}\right)
\leq \int_{|x|>a} V(x)\mu(\mathrm{d}x) \leq \beta.
\end{equation*}
Since $V(x)\rightarrow\infty$ when $\vert x \vert \rightarrow\infty$, we see that, for every $\varepsilon>0$, we can
choose $a$ large enough such that $\mu\left(\{ |x|>a \}\right) \leq \varepsilon$ uniformly in $\mu \in
\mathcal{P}_\beta(\mathbb{R}^d;V)$.
\end{proof}

\begin{propo} \label{compactness} There exists $\beta>1$ such
that $\mu_t \in \mathcal{P}_\beta(\mathbb{R}^d;V)$ for all $t$ large enough, $\mathbb{P}_{x,r,\mu}$-a.s.
\end{propo}
\noindent The proof is postponed to Section 6. Combined with Proposition \ref{P_beta}, it implies that the family $(\mu_t, t\geq 0)$ is a.s. tight.

\subsubsection{Uniform estimates on the semigroup}
A priori, it is not obvious (in a non-compact space), that the semigroup $(P_t^\mu)$ admits a (uniform) spectral gap. But this is true here. We will indeed prove a stronger result: $(P_t^\mu)$ is uniformly ultracontractive, i.e. it is uniformly bounded as an operator from $L^2(\Pi(\mu))$ to $L^\infty$. Section 5 will be devoted to those uniform properties of the family of semigroups ($P^\mu_t; t,\mu)$.

\begin{propo}\label{ultra}
The family of semigroups $(P_t^\mu, t\geq 0, \mu \in
\mathcal{P}_\beta(\mathbb{R}^d;V))$ is uniformly ultracontractive: there exists $c>0$ independent from $\mu$ such that for all $1\geq t> 0$ and $\mu \in \mathcal{P}_\beta(\mathbb{R}^d;V)$, we have
\begin{equation}
\vert \vert P_t^\mu\vert \vert_{2\rightarrow\infty} := \sup_{f\in \mathcal{C}^\infty (\mathbb{R}^d;V) \backslash \{0\}} \frac{||P_t^\mu f||_\infty}{\|f\|_{2,\mu}} \leq \exp{\left(ct^{-\delta/(\delta -1)}\right)}.
\end{equation}
\end{propo}

\begin{coro}\label{logSob} The family of measures $\left(e^{-2W*\mu(x)}\gamma(\mathrm{d}x), \mu \in \mathcal{P}_\beta(\mathbb{R}^d;V)\right)$, satisfies a logarithmic Sobolev inequality and there exists a uniform spectral gap for the family of measures $(\Pi(\mu), \mu \in \mathcal{P}_\beta(\mathbb{R}^d;V))$. It corresponds to: $\exists C, C_1, C_2 >0$, independent of $\mu$, such that $\forall f\in
\mathcal{C}^\infty(\mathbb{R}^d;V)$:
\begin{eqnarray*}
&i)& \int f^2 \log{\left(\frac{f^2}{\vert \vert f \vert
\vert_{2,\mu}}\right)}e^{-2W*{\mu}}\mathrm{d}\gamma
\leq C_2 \int |\nabla f|^2 e^{-2W*{\mu}}\mathrm{d}\gamma.\\
&ii)& \int f^2 \mathrm{d}\Pi(\mu) - \left(\int f \mathrm{d}\Pi(\mu) \right)^2 \leq C_1 \int |\nabla f|^2 \mathrm{d}\Pi(\mu).
\end{eqnarray*}
Furthermore,  $\forall t \geq 0, \, \vert \vert P_t^\mu (K_\mu
f)\vert \vert_{2,\mu} \leq e^{-t/C}\vert \vert K_\mu
f\vert\vert_{2,\mu}$.
\end{coro}
\begin{proof}
i) When a semigroup has the ultracontractivity property, then it is hypercontractive. As the uniform hypercontractivity is equivalent to the uniform Sobolev logarithmic inequality, we conclude. 

ii) Rothaus \cite{rot} has proved that if a measure satisfies a
logarithmic Sobolev inequality with constant $c$, then it also satisfies a Poincar\'e inequality with constant $1/c$. Moreover, satisfying a Poincar\'e inequality is equivalent to the
existence of a spectral gap. We easily bound $Z(\mu)\geq
\int e^{-2\kappa (V+\beta)} \mathrm{d}\gamma$.

From ii), we find the following estimate on the semigroup $(P_t^\mu)_{t\geq 0}$ (see Bakry \cite{bak}): there exists a uniform (in $\mu$) constant $C>0$ such that for all $t\geq 0$, $f\in \mathcal{C}^\infty(\mathbb{R}^d;V)$, we get $$\vert \vert P_t^\mu (K_\mu f)\vert \vert_{2,\mu} \leq e^{-t/C}\vert \vert K_\mu f\vert \vert_{2,\mu}.$$
\end{proof}

\subsection{The $\omega$-limit set}
\subsubsection{Main results}
We will show that the time-changed process $\mu_{h(t)}$ (and not $\mu_t$) is an asymptotic pseudotrajectory for $\Phi$, where $h$ is the deterministic time-change defined by
\begin{equation}
h(t) := r(e^t -1) \, \forall t\geq 0.
\end{equation}
It comes from the normalization of the occupation measure $\mu_t$. The factor $(r+t)^{-1}$ disappears while considering $$\frac{\mathrm{d}}{\mathrm{d}t}\mu_{h(t)} = \delta_{X_{h(t)}} -\mu_{h(t)}. $$
\begin{theo}\label{PTA} Under $\mathbb{P}_{x,r,\mu}$, the
function $t\mapsto \mu_{h(t)}$ is almost surely an asymptotic pseudotrajectory for $\Phi$.
\end{theo}
\noindent The proof is given in Section 6. This theorem enables us to describe the limit set of $(\mu_t)$:

\begin{coro}\label{attractor}
$\mathbb{P}_{x,r,\mu}$-a.s., $\omega(\mu_t, t\geq 0)$ is weakly compact, invariant by $\Phi$ and attractor-free. It is also contained in the convex hull of the image of $\Pi$.
\end{coro}
\noindent An attractor-free set is a set that contains no (sub-)attractor (other than itself). The exact definition  will be given later, in Section 6.

\begin{theo}\label{limitset} Assume that $W$ is symmetric.
Then, $\mathbb{P}_{x,r,\mu}$-a.s., the $\omega-$limit set of $(\mu_t, t\geq 0)$ is a connected subset of the fixed points of $\Pi$.
\end{theo}
\noindent The proof is given in Section 6. It immediately implies the following

\begin{coro}\label{converge} Suppose that $W$ is symmetric. If
$\Pi$ admits only finitely many fixed points, then $\mathbb{P}_{x,r,\mu}$-a.s., $(\mu_t;t\geq 0)$ converges to one of them.
\end{coro}

\subsubsection{A sufficient condition for the global convergence}
For a symmetric $W$, we introduce the \textit{free energy} (up to a multiplicative constant) corresponding to the ODE studied
\begin{equation}\label{free_energy}
\mathcal{F}(\mu) := \int_{\mathbb{R}^d} \log\left( \frac{\mathrm{d}\mu}{\mathrm{d}\gamma} \right)\mathrm{d}\mu + \int_{\mathbb{R}^d \times \mathbb{R}^d} W(x,y) \mu(\mathrm{d}x) \mu(\mathrm{d}y).
\end{equation}
This functional is the sum of an entropy and an interacting energy term. The competition between them can imply the existence of a unique minimizer for $\mathcal{F}$ (see \cite{cedric}).

\begin{theo}\label{conv}
Suppose that $W$ is symmetric and for all $y\in \mathbb{R}^d$, the function $x\mapsto V(x)+W(x,y)$ is strictly convex. Then there exists a unique probability measure $\mu_\infty$ such that $\underset{t\rightarrow\infty}{\lim}\mu_t = \mu_\infty \, \mathbb{P}_{x,r,\mu}- \text{a.s.}$
\end{theo}

\begin{proof}
Under our hypothesis, McCann has proved in \cite{mc} that $\mathcal{F}$ has a unique critical point $\mu_\infty$, which is a unique global minimum. It is also the unique fixed point of $\Pi$. So, $\underset{t\rightarrow\infty}{\lim}\mu_t = \mu_\infty \, \mathbb{P}_{x,r,\mu}- \text{a.s.}$
\end{proof}

\section{Study of the dynamical system $\Phi$}
\subsection{Differentiating functions of probability measures}
Here we endow the space $\mathcal{P}(\mathbb{R}^d;V)$ with a
structure of infinite-dimensional differentiable manifold. This
structure will be used only for differentiating functions defined
on $\mathcal{P}(\mathbb{R}^d;V)$, which will be needed in the study of
the semiflow and below in Section 5.\medskip

For any $\mu \in \mathcal{P}(\mathbb{R}^d;V)$ we consider the set
$\mathcal{C}^p(\mu)$ ($p \geq 1$) of (germs of) curves defined on
some neighborhood of zero $(-\varepsilon, \varepsilon)$ with
values in $\mathcal{P}(\mathbb{R}^d;V)$, passing through $\mu$ at
time zero and that are of class $\mathcal{C}^p$ when they are
considered as functions with values in the Banach space
$\mathcal{M}(\mathbb{R}^d;V)$. Now we say that a function
$\phi:\mathcal{P}(\mathbb{R}^d;V)\rightarrow \mathbb{R}$ is of
class $\mathcal{C}^p$ if for any $\mu \in
\mathcal{P}(\mathbb{R}^d;V)$ and any curve $f \in
\mathcal{C}^p(\mu)$ the real function $\phi \circ f$ is of class
$\mathcal{C}^p$. This will enable to define the differential of
such a function $\phi$. For any $\mu$ the tangent space at $\mu$
to $\mathcal{P}(\mathbb{R}^d;V)$ can be identified with the space
$\mathcal{M}_0(\mathbb{R}^d;V)$ of zero-mass measures in
$\mathcal{M}(\mathbb{R}^d;V)$. The differential is then the linear
operator:
\begin{equation}
D \phi(\mu) \cdot \nu = \frac{\mathrm{d}}{\mathrm{d}t}\phi(\mu +
t\nu)\vert_{t=0}, \quad \nu \in \mathcal{M}_0(\mathbb{R}^d;V).
\end{equation}
The same definition can apply for functions with values in a
Banach space or even in $\mathcal{P}(\mathbb{R}^d;V)$. As an
example (to be used!), the maps $\mu\mapsto W*\mu(x)$ (for any
point $x$) and $\Pi$ (applying the Lebesgue theorem) are
$\mathcal{C}^\infty$.

\subsection{Existence of the semiflow}
We first prove the local existence of the semiflow and then give
sufficient conditions on the potentials for non-explosion. We
recall the equation:
\begin{equation}\label{eq_edo}
\Phi_t(\mu) = e^{-t} \mu + e^{-t}\int_0^t e^s \Pi(\Phi_s(\mu))\,
\mathrm{d}s.
\end{equation}
For proving the local existence of a solution, since
$\mathcal{P}(\mathbb{R}^d;V)$ is not a vector space, we will
proceed directly by approximation. The following lemma is helpful
in order to find a good security cylinder.

\begin{lemma}\label{Pibound}
For any $\beta>1$, the application $\Pi$ restricted to
$\mathcal{P}_\beta(\mathbb{R}^d;V)$ is bounded and Lipschitz.
\end{lemma}
\begin{proof}
First we need to show that $\mu\mapsto Z(\mu)$ is bounded from
below. For $\mu \in \mathcal{P}_\beta(\mathbb{R}^d;V)$, Lemma 
\ref{majW} asserts that $W*\mu(x) \leq 2\kappa \beta V(x)$. So we
get:
$$Z(\mu)=\int_{\mathbb{R}^d} e^{-2W*\mu(x)} \gamma(\mathrm{d}x)\geq \int_{\mathbb{R}^d}
e^{-4\kappa\beta V(x)} \gamma(\mathrm{d}x)$$ and thus we have the
following bound for $\Pi(\mu)$:
\begin{equation}\label{Cbeta}
\vert \vert \Pi(\mu) \vert\vert_V \leq \left( \int_{\mathbb{R}^d}
e^{-4\kappa\beta V(x)} \gamma(\mathrm{d}x)\right)^{-1}
\int_{\mathbb{R}^d} V(x) \gamma(\mathrm{d}x) =: C_\beta.
\end{equation}
We know that $\Pi$ is $\mathcal{C}^\infty$ on
$\mathcal{P}(\mathbb{R}^d;V)$ with the strong topology. Its
differential (at $\mu$) is the continuous linear operator
$D\Pi(\mu): \mathcal{M}_0(\mathbb{R}^d;V) \rightarrow
\mathcal{M}_0(\mathbb{R}^d;V)$ defined by
\begin{equation}\label{diffPi}
D\Pi(\mu)\cdot \nu(\mathrm{d}x):= -2\left(W*\nu(x)
-\int_{\mathbb{R}^d}
W*\nu(y)\Pi(\mu)(\mathrm{d}y)\right)\Pi(\mu)(\mathrm{d}x).
\end{equation}
Fix $\nu \in \mathcal{M}_0(\mathbb{R}^d;V)$. Since $\vert W*\nu(x)
\vert \leq 2\kappa\vert \vert \nu \vert\vert_V V(x)$, we find that
\begin{equation*}
\vert \vert D\Pi(\mu)\cdot \nu \vert \vert_V \leq
4\kappa(1+C_\beta) \vert \vert \nu \vert\vert_V\int_{\mathbb{R}^d}
V^2(x)\Pi(\mu)(\mathrm{d}x).
\end{equation*}
But for $\mu \in \mathcal{P}_\beta(\mathbb{R}^d;V)$, the same
computation used for the bound of $\Pi(\mu)$ enables to control
the last integral, hence we get a bound (call it $C'_\beta$) on
the differential and $\Pi$ is Lipschitz as stated.
\end{proof}

\begin{propo} For all $\mu \in \mathcal{P}(\mathbb{R}^d;V)$ the ODE has a local
solution. This defines a $C^\infty$ semiflow $\Phi$ for the strong topology.
\end{propo}
\begin{proof}
Let $\mu$ belong to $\mathcal{P}_\beta(\mathbb{R}^d;V)$ where we
choose $\beta>2\vert\vert \mu \vert\vert_V$. We introduce the
classic Picard approximation scheme:
\begin{equation*}
\left\{%
\begin{array}{ll}
\mu^{(0)}_t := \mu,\\
\mu^{(n)}_t := e^{-t}\mu + \int_0^t
e^{s-t}\Pi\left(\mu^{(n-1)}_s\right)\mathrm{d}s.
\end{array}%
\right.
\end{equation*}
We set $\varepsilon$ small enough such that $\vert\vert \mu
\vert\vert_V + (1-e^{-\varepsilon})C_\beta \leq \beta$ and
$\varepsilon C'_\beta < 1$ where both constants were defined in
Lemma \ref{Pibound}. Then, for all $n$, $\mu^{(n)}_t$ is defined
and belongs to $\mathcal{P}_\beta(\mathbb{R}^d;V)$, which makes
$[0,\varepsilon)\times\mathcal{P}_\beta(\mathbb{R}^d;V)$ a good
security cylinder. We have, for $t<\varepsilon$,
\begin{equation*}
\vert \vert \mu^{(n+1)}_t - \mu^{(n)}_t \vert \vert_V \leq
(1-e^{-\varepsilon})C'_\beta\sup_{t<\varepsilon} \vert \vert
\mu^{(n)}_t - \mu^{(n-1)}_t \vert \vert_V.
\end{equation*}
Now the series with general term $\sup_{t<\varepsilon} \vert \vert
\mu^{(n+1)}_t - \mu^{(n)}_t \vert \vert_V$ converges and thus the
sequence of functions $\mu^{(n)}$ is Cauchy for the topology of
uniform convergence. Since $\mathcal{P}(\mathbb{R}^d;V)$ is
complete, we have successfully built a solution on
$[0,\varepsilon)$. There remains to show that the semiflow is
smooth. We have seen that the map $\Pi$ is $\mathcal{C}^\infty$
for the strong topology. By induction, every Picard approximation
$\mu\mapsto \mu^{(n)}_t$ is $\mathcal{C}^\infty$ and it is enough
to take the limit uniformly in $\mu$ on
$\mathcal{P}_\beta(\mathbb{R}^d;V)$ to conclude.
\end{proof}

\begin{defn} A subset $A$ of $\mathcal{P}(\mathbb{R}^d;V)$ is \textit{positively invariant} (negatively invariant, \textrm{invariant}) for $\Phi$ provided $\Phi_t(A) \subset A$ ($A \subset \Phi_t(A)$, $\Phi_t(A) = A$) for all $t\geq 0$.
\end{defn}

\begin{propo}\label{exist-flow}
Whenever that $W$ is either symmetric or bounded in the second variable ($W(x,y)\leq \kappa V(x)$), then the semiflow $\Phi$ does not explode.
\end{propo}
\begin{proof}
The case where $W(x,y)$ is bounded in $y$ is easy. Since we have $W(x,y)\leq \kappa V(x)$, mimicking the proof of Lemma \ref{Pibound} enables to show that $\Pi$ is globally bounded (call $C$ the upper bound). This means that $\Phi_t(\mu)$ remains in the space $\mathcal{P}_C(\mathbb{R}^d;V)$, therefore it cannot explode.

Let us now assume that $W$ is symmetric. We point out that the free energy \eqref{free_energy} is not a Lyapunov function for \eqref{eq_edo} because in general the measure
$\Phi_t(\mu)$ is not absolutely continuous with respect to the
Lebesgue measure and so, $\mathcal{F}(\Phi_t(\mu)) = \infty$.
Thus, consider the Lyapunov function $\mathcal{E}(\mu)
:=\mathcal{F}(\Pi(\mu))$. Indeed, $\mathcal{F}$ restricted to
absolutely continuous probability measures is a
$\mathcal{C}^\infty$ function for the strong topology ($V$-norm).
We compute (thanks to the symmetry of $W$) for $\nu \in \mathcal{M}_0(\mathbb{R}^d;V)$
\begin{equation}\label{diffF}
D\mathcal{F}(\mu)\cdot \nu = \int_{\mathbb{R}^d}\left[ \log\left(
\frac{\mathrm{d}\mu }{ \mathrm{d}\gamma}(x)\right) +
2W*\mu(x)\right] \mathrm{d}\nu(x).
\end{equation}
But we recall that $\Pi$ is $\mathcal{C}^\infty$ and equation \eqref{diffPi}. So, since by composition $D\mathcal{E}(\mu)\cdot \nu = D\mathcal{F}(\Pi(\mu))\circ D\Pi(\mu)\cdot \nu$, we obtain
\begin{eqnarray*}
D\mathcal{E}(\mu)\cdot \nu = -4\int_{\mathbb{R}^d} \left(W*\Pi(\mu) -W*\mu\right) \left(W*\nu -\int_{\mathbb{R}^d}W*\nu\, \mathrm{d}\Pi(\mu)\right) \mathrm{d}\Pi(\mu).
\end{eqnarray*}
It remains to choose $\nu = \Pi(\mu)-\mu$ in order to get
\begin{eqnarray*}
\frac{1}{4}\frac{\mathrm{d}}{\mathrm{d}t} \mathcal{E}(\Phi_t(\mu)) = - \int_{\mathbb{R}^d} (W*\nu)^2 \mathrm{d}\Pi(\mu) + \left(\int_{\mathbb{R}^d}W*\nu\, \mathrm{d}\Pi(\mu)\right)^2\leq 0.
\end{eqnarray*}
Therefore, for all $c>0$, the sets $\{\mu; \mathcal{E}(\mu) \leq c\}$ are positively invariant. As they are (weakly) compact,
the semiflow cannot explode.
\end{proof}

\subsection{An important set}
We introduce here a crucial object for the analysis of the dynamical system $\Phi$. Let
\begin{equation}
\mathrm{Im} (\Pi) := \left\{\Pi(\mu); \mu \in \mathcal{P}(\mathbb{R}^d;V)\right\},
\end{equation}
and denote its convex hull by $\widehat{\mathrm{Im}(\Pi)}$.

\begin{propo} \label{im}
$\widehat{\mathrm{Im}(\Pi)}$ is a positively invariant set for the
semiflow $\Phi$ which contains every negatively invariant bounded
subset of $\mathcal{P}(\mathbb{R}^d;V)$.
\end{propo}
\begin{proof}
To prove the result, it is enough to show for every $\mu \in
\mathcal{P}(\mathbb{R}^d;V)$ and every $t \geq 0$ the inequality
\begin{equation}
\mathrm{d}_V\left(\Phi_t(\mu),\widehat{\mathrm{Im}(\Pi)}\right)
\leq e^{-t}
\mathrm{d}_V\left(\mu,\widehat{\mathrm{Im}(\Pi)}\right),
\end{equation}
where $\mathrm{d}_V(\mu, X) := \inf \{\|\mu-\nu\|_V; \nu \in X\}$.
But this inequality directly results from the Jensen inequality
applied to the convex combination $\Phi_t(\mu) = e^{-t} \mu + e^{-t}\int_0^t e^s \Pi(\Phi_s(\mu))\, \mathrm{d}s$ and to
the convex map $\mu \mapsto \mathrm{d}_V(\mu,\widehat{\mathrm{Im}(\Pi)})$.
\end{proof}

\subsection{Fixed points of $\Pi$}
We show how the free energy functional $\mathcal{F}$ \eqref{free_energy} can help to find the fixed points of $\Pi$.
\begin{propo}\label{min} Suppose that $W$ is symmetric. Then the fixed points of $\Pi$
are the minima of $\mathcal{F}$.
\end{propo}
\begin{proof}
Equation \eqref{diffF} readily implies that, we have $D\mathcal{F}(\mu) \cdot \nu = 0$ for all $\nu \in \mathcal{M}_0(\mathbb{R}^d;V)$ if and only if $\mu = \Pi(\mu)$. So, the fixed points of $\Pi$ are the critical points of $\mathcal{F}$. Moreover, $\mathcal{F}$ is a convex functional. Indeed, it is a $\mathcal{C}^\infty$ functional (on the set of absolutely continuous measures), with second differential $D^2 \mathcal{F}(\mu)$. Let $\nu_1,\nu_2 \in \mathcal{P}(\mathbb{R}^d;V)$. We have:
$$D^2 \mathcal{F}(\mu)\cdot (\nu_1,\nu_2) = \int_{\mathbb{R}^d} \nu_1(x)\nu_2(x) \mu(x)^{-1} \gamma(x) \mathrm{d}x + \int_{\mathbb{R}^d}\int_{\mathbb{R}^d} W(x,y) \nu_1(\mathrm{d}x)\nu_2(\mathrm{d}y)$$ and the convexity is a consequence of the nonnegativity of $W$. It then implies that $\mu = \Pi(\mu)$ is a minimum for $\mathcal{F}$.
\end{proof}

\begin{lemma}\label{lem_pi}
Whenever that $W$ is either symmetric or bounded in the second variable, then the set of the fixed points of $\Pi$, $\left\{\mu \in \mathcal{P}(\mathbb{R}^d); \Pi(\mu) =\mu \right\}$, is a nonempty compact (for the weak topology) subset of $\mathcal{P}(\mathbb{R}^d;V)$.
\end{lemma}
\begin{proof}
Suppose first that $W$ is symmetric. We will again use the free energy. Let $m := \inf\{ \mathcal{E}(\mu); \mu \in \mathcal{P}(\mathbb{R}^d;V)\}$. There exists a sequence of probability measures $(\mu_n)$ absolutely continuous with respect to the Lebesgue measure such that $m \leq \mathcal{E}(\mu_n) \leq m+ 1/n$. But, as proved in Proposition \ref{exist-flow}, for any $c>0$, the set $\{\mu; \mathcal{E}(\mu) \leq c\}$ is compact. So, we extract a subsequence $(\mu_{n_k})$, converging (for the weak topology) to $\mu_\infty$. As $\mu\mapsto W*\mu$ and $\mu \mapsto \Pi(\mu)$ are two continuous functions, $\mu \mapsto \mathcal{E}(\mu)$ is also weakly continuous and so $\mathcal{E}(\mu_\infty)=m$. We conclude by Proposition \ref{min}.

Suppose now that $W$ is bounded in $y$: $W(x,y)\leq \kappa V(x)$. We have proved in Lemma \ref{exist-flow} that $\Pi(\mu)$ maps (weakly) continuously the compact convex space $\mathcal{P}_C(\mathbb{R}^d;V)$ into itself. The Leray-Schauder fixed point theorem then ensures that the set $\left\{\mu \in \mathcal{P}(\mathbb{R}^d;V); \Pi(\mu) =\mu\right\}$ is nonempty. 
\end{proof}

\section{Study of the family of semigroups $(P_t^\mu, t\geq 0, \mu\in\mathcal{P}_\beta(\mathbb{R}^d;V))$}

In this section, we introduce two crucial functional inequalities
for the family of semigroups $P_t^\mu$: the spectral gap and the
ultracontractivity. Since we consider these semigroups altogether
for all the measures $\mu\in\mathcal{P}_\beta(\mathbb{R}^d;V)$, we
will prove that the constants involved in those properties are
uniform in $\mu$. The notion of ultracontractivity and its
relation to the analysis of Markov semigroups were first studied
by Davies and Simon \cite{daS} and recently by R\"ockner \& Wang
\cite{roWa} for more general diffusions. The need for
ultracontractivity will impose some kind of boundedness on the
convolution term in the SDE that cannot be easily removed.
Finally, thanks to these properties, we compute several estimates
that prepare the proof of Section 6.\medskip

For any $\mu\in\mathcal{P}(\mathbb{R}^d;V)$, let as usual $L^2(\Pi(\mu))$ denote the Lebesgue space of Borel square-integrable functions with respect to the measure $\Pi(\mu)$. We remark that the space depends on $\mu$, but we will
consider mainly the subspace $\mathcal{C}^0(\mathbb{R}^d;V)\subset L^2(\Pi(\mu))$. We denote $$(f,g)_\mu := \int_{\mathbb{R}^d} f(x) g(x) \Pi(\mu)(\mathrm{d}x)$$ the inner product on this space and $\vert \vert.\vert \vert_{2,\mu}$ the associated norm. We introduce two operators : $Q_\mu$ is the ``inverse" of $A_\mu$, defined for any function $f$ by
\begin{eqnarray}
Q_\mu f := \int_0^\infty \left(P_t^\mu f - \Pi(\mu)f\right)
\mathrm{d}t
\end{eqnarray}
and $K_\mu$ is the orthogonal projector defined by
\begin{eqnarray}
K_\mu f := f - \Pi(\mu) f.
\end{eqnarray}
They are linked together by the following relation $\forall f \in \mathcal{C}^\infty(\mathbb{R}^d;V),$ 
\begin{eqnarray*}
A_\mu \circ Q_\mu (f) = Q_\mu \circ A_\mu (f) = -K_\mu f.
\end{eqnarray*}

\begin{rk} The integrability of $(P_t^\mu f - \Pi(\mu)f)$ will come from
the uniform spectral gap obtained in Corollary \ref{logSob}.
\end{rk}

\subsection{Uniform ultracontractivity}
To prove that the family of semigroups $\left(P_t^\mu,t\geq 0,\mu
\in \mathcal{P}_\beta(\mathbb{R}^d;V)\right)$ is uniformly
ultracontractive, we will rely on the following result of
R\"ockner \& Wang:
\begin{lemma}\label{rwang} (\cite{roWa} corollary 2.5) Let $(P_t,t\geq 0)$ be a
Markov semigroup, with infinitesimal generator $A := \frac{1}{2}
\Delta - (\nabla U,\nabla)$, and $\nabla^2 U \geq -K$. Assume that
there exists a continuous increasing map $\chi: \mathbb{R}_+
\mapsto \mathbb{R}_+ \setminus \{0\}$ such that
\begin{enumerate}
    \item $\underset{r\rightarrow \infty}{\lim} \frac{\chi(r)}{r} = \infty$,
    \item the mapping $g_\chi(r) := r \chi(m \log r)$ is convex on $[1,\infty)$ for any $m >0$,
    \item $A |x|^2 \leq b -\chi(|x|^2)$ for some $b>0$.
\end{enumerate}
Then $P_t$ has a unique invariant probability measure. If $\int_2^\infty \frac{\mathrm{d}r}{r\chi(m \log r)}<\infty$, $m>0$,
then $P_t$ is ultracontractive.\\
If moreover $\chi(r) = \chi r^\delta$, with $\chi>0, \delta>1$, then there exists $c=c(b,\chi)>0$ such that for all $t\in
(0,1],\quad \vert \vert P_t \vert \vert_{2\rightarrow \infty} \leq \exp{\left(ct^{-\delta/(\delta -1)}\right)}.$
\end{lemma}

\begin{propa} The family of semigroups $(P_t^\mu, t\geq 0, \mu \in
\mathcal{P}_\beta(\mathbb{R}^d;V))$ is uniformly ultracontractive: we have for all $1\geq t> 0$ and $\mu
\in \mathcal{P}_\beta(\mathbb{R}^d;V)$
\begin{equation}
\vert \vert P_t^\mu\vert \vert_{2\rightarrow\infty}:= \sup_{f\in \mathcal{C}^\infty (\mathbb{R}^d;V)\backslash \{0\}} \frac{\|P_t^\mu f\|_\infty}{\|f\|_{2,\mu}} \leq \exp{\left(ct^{-\delta/(\delta -1)}\right)},
\end{equation}
where $c>0$ is independent from $\mu$.
\end{propa}
\begin{proof}
We apply Lemma \ref{rwang} with $U:=V + W*\mu$ and find that each
$(P_t^\mu)_{t\geq 0}$ is ultracontractive. Indeed, the conditions
(\ref{curvature}) and the growth condition on $V$, all together imply that there exist $a,b>0$
such that for $\mu \in \mathcal{P}_\beta(\mathbb{R}^d;V)$
$$ A_{\mu} |x|^2 = d -2 (\nabla W*{\mu}(x),x) - 2(\nabla V(x),x)
\leq b- a |x|^{2\delta}.$$ As we let $\chi(r) := r^\delta$, with $\delta>1$,
we find that the constant $c$ is uniform in $\mu$. Thus, we have the uniform
ultracontractivity.
\end{proof}
We recall that, as a consequence of Proposition \ref{ultra}, there exists $C>0$, uniform in $\mu$, such that $\forall t \geq 0,$ $$\vert \vert P_t^\mu (K_\mu f)\vert \vert_{2,\mu} \leq e^{-t/C}\vert \vert K_\mu f\vert\vert_{2,\mu}.$$
We are now able to derive some usefull bounds on the operator $Q_\mu$.
\begin{propo}\label{Q_muC1} For all $\varepsilon>0$, there exists a positive
constant $K(\varepsilon)$ such that for all $\mu \in \mathcal{P}_\beta(\mathbb{R}^d;V)$,
$x\in \mathbb{R}^d$, $f\in \mathcal{C}^0(\mathbb{R}^d;V)$:
\begin{equation}
\vert Q_\mu f(x)\vert \leq (\varepsilon V(x) + K(\varepsilon))
\vert \vert f\vert \vert_V.
\end{equation}
\end{propo}
\begin{proof}
Let $t_0 \in (0,1]$ (we will choose it precisely
later). We have: $$|Q_\mu f(x)|\leq \int_0^\infty \vert P_t^\mu
(K_\mu f)(x)\vert \mathrm{d}t = \int_0^{t_0} \vert P_t^\mu (K_\mu
f)(x)\vert \mathrm{d}t + \int_{t_0}^\infty \vert P_t^\mu (K_\mu
f)(x)\vert \mathrm{d}t.$$ We begin to work with the second
right-hand term. By use of the uniform ultracontractivity and the
uniform spectral gap, we have
\begin{eqnarray*}
\int_{t_0}^\infty \vert P_t^\mu (K_\mu f)(x)\vert \mathrm{d}t &=&
\int_{0}^\infty \vert P_{t_0}^\mu P_t^\mu (K_\mu f)(x)\vert\mathrm{d}t,\\
\int_{t_0}^\infty \vert P_t^\mu (K_\mu f)(x)\vert \mathrm{d}t
&\leq& \exp{\left(ct_0^{-\delta/(\delta -1)}\right)}
\int_{0}^\infty e^{-t/C_1} \mathrm{d}t \vert\vert K_\mu
f\vert\vert_{2,\mu}.
\end{eqnarray*}
As $K_\mu$ is an orthogonal projector, $\|K_\mu f\|_{2,\mu} \leq
\|f\|_{2,\mu} \leq \left(\int V^2
\mathrm{d}\Pi(\mu)\right)^{1/2}\|f\|_V$, and we get
$$\int_{t_0}^\infty \vert P_t^\mu (K_\mu f)(x)\vert
\mathrm{d}t\leq C_1\exp{\left(ct_0^{-\delta/(\delta -1)}\right)}
\left(\int V^2 \mathrm{d}\Pi(\mu)\right)^{1/2} \vert\vert
f\vert\vert_V.$$ We now have to work with the first right-hand
term. We have with the naive estimate \eqref{Xmu_bound}: $$\vert
P_t^\mu f(x)\vert \leq ||f||_V P_t^\mu V(x) \leq ||f||_V \mathbb{E} \mathcal{E}_{\mu}(X^\mu_t)\leq \mathcal{E}_\mu(x)e^{Dt} \|f\|_V.$$ Since $\mathcal{E}_\mu(x) \leq 3\kappa\beta V(x)$, we finally find
$$\int_0^{t_0}\vert P_t^\mu K_\mu f(x) \vert \mathrm{d}t \leq
4\kappa \beta \int_0^{t_0} e^{Dt} \mathrm{d}t \vert\vert
f\vert\vert_V V(x).$$ Now we choose $t_0$ small enough such that
$4\kappa\beta \int_0^{t_0} e^{Dt} \mathrm{d}t \leq \varepsilon$ to conclude.
\end{proof}

\begin{propo}\label{gradQ_mu} For all $\varepsilon>0$, there exists $K_1(\varepsilon)>0$ such that for all $\mu \in
\mathcal{P}_\beta(\mathbb{R}^d;V), x\in \mathbb{R}^d, f\in
\mathcal{C}^\infty (\mathbb{R}^d;V)$, we have $Q_\mu f \in
\mathcal{C}^1(\mathbb{R}^d)$ and:
\begin{equation}
\vert \nabla Q_\mu f(x) \vert \leq (\varepsilon V(x)+ K_1(\varepsilon)) \vert \vert f\vert \vert_V.
\end{equation}
\end{propo}
\begin{proof}
Suppose that $f$ is smooth. We introduce two operators: the `carr\'e du champ'' $\Gamma(f)= |\nabla f|^2$ and $\Gamma_2^\mu(f) = |\nabla^2 f|^2 + (\nabla f,\nabla^2(V+W*\mu) \nabla f)$. As we have the curvature condition \eqref{curvature}, we get (for the curvature $M\in\mathbb{R}$) $\Gamma_2^\mu(f) \geq M \Gamma(f)$. The $\Gamma_2$-criterion implies the following (see Ledoux \cite{led} p22), 
$\forall f\in \mathcal{C}^\infty(\mathbb{R}^d;V)$, $\forall t>0$
\begin{equation}\label{nablaP}
\vert \nabla P_t^\mu(K_\mu f)\vert^2 \leq \frac{M}{e^{2Mt}-1} \vert P_t^\mu (K_\mu f)^2\vert.
\end{equation}
Indeed, one can show that 
\begin{eqnarray*}
\int_{t_0}^\infty \vert \nabla P_t^\mu (K_\mu f)(x)\vert \mathrm{d}t &\leq & \sqrt{\frac{M}{e^{2Mt_0}-1}} \int_0^\infty \left[ P_{t_0} (P_t^\mu (K_\mu f))^2(x) \right]^{1/2} \mathrm{d}t\\
&\leq & C(t_0)\vert\vert f\vert\vert_V \left(\int V^4 \mathrm{d}\Pi(\mu)\right)^{1/4}
\end{eqnarray*}
where $C(t_0) = 2C_1\sqrt{\frac{M}{e^{2Mt_0}-1}}\exp{\{ct_0^{-\delta/(\delta -1)}/2\}}$. Finally, similarly to Proposition \ref{non-expl}, one proves that $$\mathbb{E}_x V^2 (X_{t}^\mu) \leq \mathcal{E}_\mu^2(x) e^{2Dt}\leq (3\kappa \beta)^2 V^2(x) e^{2Dt}.$$ 
\end{proof}

\subsection{Regularity with respect to the measure $\mu$}
First, consider the Banach space $\mathcal{B}$ of bounded linear operators from $\mathcal{C}^\infty(\mathbb{R}^d;V) \subset L^2(\gamma)$, endowed with the norm $\| f \|_{2,\mu,1} := \| f \|_{2,\mu} + \| A_\mu f \|_{2,\mu}$, to the same space equipped with the standard quadratic norm. We endow $\mathcal{B}$ with the operator norm. Then, $A_\mu$ obviously belongs to the closed subset of $\mathcal{B}$ consisting in operators $A$ such that $A 1 =0$. This allows us to state and prove the following:
\begin{propo}
The mappings $\mu\mapsto A_\mu$ and $\mu\mapsto K_\mu$ are $\mathcal{C}^\infty$. For any function $f \in \mathcal{C}^\infty (\mathbb{R}^d;V)$, the application $\mu\mapsto Q_\mu f$ is $\mathcal{C}^\infty$ for the strong topology and we have for
the differentials (for any $\mu \in \mathcal{P}(\mathbb{R}^d;V)$, $\nu \in \mathcal{M}_0(\mathbb{R}^d;V)$):
\begin{eqnarray*}
D(A_\mu f) \cdot \nu &=& - (\nabla W* \nu, \nabla f); \\
D(K_\mu f) \cdot \nu &=& - \left(D \Pi(\mu) \cdot \nu \right)(f); \\
D(Q_\mu f) \cdot \nu &=& \left(D \Pi(\mu) \cdot \nu\right)(Q_\mu f) + Q_\mu(\nabla W* \nu, \nabla Q_\mu f).
\end{eqnarray*}
\end{propo}
\begin{proof}
We already know that $\mu \mapsto W*\mu $ and $\Pi$ are
$\mathcal{C}^\infty$; so there is nothing to prove in case of
$A_\mu$ or $K_\mu$. To look at $Q_\mu$, we need to consider the resolvent operator of $P^\mu_t$:
\begin{equation}
R^\mu_\lambda := \int_0^\infty e^{-\lambda t}P^\mu_t \mathrm{d}t =
(\lambda - A_\mu)^{-1}, \, \forall \lambda > 0.
\end{equation}
For $\lambda >0$, we define the following approximation of $Q_\mu$
\begin{equation}
Q_\mu(\lambda) := \int_0^\infty e^{-\lambda t}P^\mu_t K_\mu
\mathrm{d}t = K_\mu (\lambda - A_\mu)^{-1}.
\end{equation}
As $\mu\mapsto K_\mu$ and $\mu \mapsto A_\mu$ are
$\mathcal{C}^\infty$, we find by composition that the map
$\mu\mapsto R^\mu_\lambda f$ is $\mathcal{C}^\infty$.

Now let $\beta >1$ and consider the measures in
$\mathcal{P}_\beta(\mathbb{R}^d;V)$. The uniform spectral gap
shows that there exist $C,C_1>0$ such that we have
\begin{equation*}
\vert \vert Q_\mu f - Q_\mu(\lambda)f\vert\vert_V \leq
\int_0^\infty \mathrm{d}t (1-e^{-\lambda t}) \vert \vert P_t^\mu
K_\mu f \vert \vert_V \leq \lambda C \vert\vert f\vert\vert_V
\int_0^\infty t e^{-tC_1} \mathrm{d}t.
\end{equation*}
Hence the convergence of $Q_\mu(\lambda)$ towards $Q_\mu$ is
uniform with respect to $\mu$ on
$\mathcal{P}_\beta(\mathbb{R}^d;V)$. As a by-product, $\mu\mapsto
Q_\mu f$ is continuous.\medskip

We have the following differential:
\begin{eqnarray*}
DQ_\mu(\lambda) \cdot \nu &=& (DK_\mu \cdot \nu)(\lambda -
A_\mu)^{-1} + K_\mu (\lambda - A_\mu)^{-1} (DA_\mu \cdot
\nu)(\lambda - A_\mu)^{-1}\\
&=& (D\Pi(\mu) \cdot \nu)(\lambda - A_\mu)^{-1} + K_\mu (\lambda
- A_\mu)^{-1} (\nabla W*\nu,\nabla)(\lambda - A_\mu)^{-1}.
\end{eqnarray*}
We will prove that each right side term of the preceding equality converges
uniformly. For the first term, we have for all $f\in
\mathcal{C}^\infty(\mathbb{R}^d;V)$:
\begin{equation*}
(D\Pi(\mu) \cdot \nu)((\lambda - A_\mu)^{-1}f) = \left(D\Pi(\mu)
\cdot \nu\right)\left(K_\mu(\lambda - A_\mu)^{-1}f\right)
\end{equation*}
and therefore
\begin{equation*}
\lim_{\lambda\rightarrow 0+} (D\Pi(\mu) \cdot \nu)((\lambda -
A_\mu)^{-1})f = \left(D\Pi(\mu) \cdot \nu\right)\left(Q_\mu
f\right)
\end{equation*}
where the convergence is uniform in $\mu$. It remains to prove the
convergence of the second term. We have
\begin{equation*}
K_\mu (\lambda - A_\mu)^{-1} (\nabla W*\nu,\nabla )((\lambda -
A_\mu)^{-1}f) = Q_\mu(\lambda) (\nabla W*\nu,\nabla
Q_\mu(\lambda)f).
\end{equation*}
If we manage to prove that $\nabla Q_\mu(\lambda)f$ converges
(uniformly in $\mu$) to $\nabla Q_\mu f$, then we are done. We
have by definition of $Q_\mu(\lambda)$: $$\nabla Q_\mu(\lambda)f =
\int_0^\infty \nabla (P_t^\mu K_\mu f) e^{-\lambda t}
\mathrm{d}t$$ and therefore $$\vert \nabla Q_\mu f- \nabla
Q_\mu(\lambda)f \vert \leq \int_0^\infty \vert \nabla (P_t^\mu
K_\mu f)\vert (1-e^{-\lambda t}) \mathrm{d}t.$$ We use inequality
(\ref{nablaP}) to prove that this family of differentials converges
uniformly with respect to $\mu$; so $\mu\mapsto Q_\mu f$ is
actually $\mathcal{C}^1$ with the differential given in the
statement of the proposition. 
\end{proof}

\begin{rk} Looking at the differential $D (Q_\mu f)$, we see that it is itself a $\mathcal{C}^1$ function of $\mu$, so by
induction it can be proved that $\mu\mapsto Q_\mu f$ is $\mathcal{C}^\infty$. We could have also proved that $\mu\mapsto P^\mu_t f$ is $\mathcal{C}^\infty$. But these results will not be needed in the remainder.
\end{rk}

\begin{coro}\label{DQmu}
For every $f\in \mathcal{C}^\infty(\mathbb{R}^d;V)$, we have the uniform inequality
$$\vert (DQ_\mu \cdot \nu) (f)(x)\vert \leq (\varepsilon V(x)^2 + K_2(\varepsilon)) \|f\|_V \|\nu\|_V.$$
\end{coro}
\begin{proof}
We have the following:
\begin{eqnarray*}
\vert (DQ_\mu \cdot \nu) (f)(x)\vert &\leq & \vert (D\Pi(\mu)\cdot
\nu)(Q_\mu f)\vert + \vert Q_\mu (\nabla W*\nu(x),\nabla Q_\mu
f(x))\vert.
\end{eqnarray*}
We will treat each of the two terms on the right side separately. If we consider the second right hand term, we find
\begin{eqnarray*}
\vert Q_\mu (\nabla W*\nu(x),\nabla Q_\mu f(x))\vert &\leq & (\varepsilon V^2(x) + K(\varepsilon))\vert \vert (\nabla W*\nu,\nabla Q_\mu f)\vert\vert_{V^2}\\
&\leq & (\varepsilon V^2(x) + K(\varepsilon))\vert \vert \nabla W*\nu\vert \vert_V \vert \vert\nabla Q_\mu f\vert\vert_{V}\\
&\leq & (\varepsilon V^2(x) + K'(\varepsilon))\vert \vert \nu\vert\vert_V \vert \vert f\vert\vert_V.
\end{eqnarray*}
We work now with the other member of the inequality.
\begin{eqnarray*}
\vert (D\Pi(\mu)\cdot \nu)(Q_\mu f) \vert &\leq & 2 \int\vert Q_\mu f(x)\vert \left| W*\nu(x) -\int W*\nu \mathrm{d}\Pi(\mu) \right| \Pi(\mu)(\mathrm{d}x)\\
&\leq & C\| f\|_V \vert \vert \nu\vert \vert_V \int (\varepsilon V(x) + K(\varepsilon)) (V(x) + 1) \Pi(\mu)(\mathrm{d}x)\\
&\leq & C'\vert\vert f\vert \vert_V \|\nu\vert \vert_V.
\end{eqnarray*}
Putting the pieces together, we are done.
\end{proof}

\section{Behavior of the occupation measure}
\subsection{Tightness of $(\mu_t,t \geq 0)$}
Thanks to the potential $V$, we manage to obtain the weaker form of compactness of the occupation measure: tightness.
\begin{propb} Let $x,r,\mu$ be given. Then there exists $\beta>1$ such that
$\mathbb{P}_{x,r,\mu}-$a.s. $\mu_t \in \mathcal{P}_\beta(\mathbb{R}^d;V)$ for all $t$ large enough.
\end{propb}
\begin{proof}
We set $\phi(t):=\int_0^t V(X_s)\mathrm{d}s$. All we need to prove
is that $\mathbb{P}_{x,r,\mu}-$a.s. $\phi(t)=O(t)$. We use again the Lyapunov functional
$\mathcal{E}_{\mu}(x)=V(x)+W*\mu(x)$. We have already shown:
\begin{eqnarray*}
\mathcal{E}_{\mu_t}(X_t) &=& \mathcal{E}_{\mu}(x) + \int_0^t (\nabla \mathcal{E}_{\mu_s}(X_s), \mathrm{d}B_s) - \int_0^t \left|\nabla \mathcal{E}_{\mu_s}(X_s)\right|^2 \mathrm{d}s\\
&+& \frac{1}{2}\int_0^t \Delta \mathcal{E}_{\mu_s}(X_s)\mathrm{d}s
+ \int_0^t \left(W(X_s,X_s) - W*\mu_s(X_s) \right)\frac{\mathrm{d}s}{r+s}.
\end{eqnarray*}
The strong law of large numbers (for martingales) implies that a.s. for
$t$ large enough we will have $\int_0^t (\nabla \mathcal{E}_{\mu_s}(X_s), \mathrm{d}B_s) \leq \frac{1}{2}\int_0^t \left| \nabla \mathcal{E}_{\mu_s}(X_s) \right|^2 \mathrm{d}s$, and therefore we get the a.s. inequality for $t$ large enough:
\begin{equation*}
\int_0^t \left| \nabla \mathcal{E}_{\mu_s}(X_s) \right|^2 \mathrm{d}s \leq 2\mathcal{E}_{\mu}(x) + \int_0^t \Delta \mathcal{E}_{\mu_s}(X_s)\mathrm{d}s + \frac{2}{r}\int_0^t W(X_s,X_s)\mathrm{d}s.
\end{equation*}
Now we want to  find an integral inequality on $\phi$. To this
aim, we will control separately each of the three terms of the
last inequality and let $\phi(t)$ appear.

i) From the growth assumption on $V$, for any $\epsilon >0$ we can
find $k_{\epsilon}>0$ such that $V \leq k_{\epsilon} + \epsilon
|\nabla V|^2$, thus by means of an integration we get a.s.
\begin{equation*}
\phi(t) \leq k_{\epsilon}t + \epsilon\int_0^t \left| \nabla \mathcal{E}_{\mu_s}(X_s) \right|^2 \mathrm{d}s.
\end{equation*}
ii) From the domination condition (\ref{domination}) on $W$, we
have $\Delta W*\mu(x) \leq \kappa (V(x) + \mu(V))$; the growth
condition (\ref{deltaV}) on $V$ ensures the inequality $\Delta V
\leq aV$ for some $a>0$; therefore we get a.s.
\begin{equation*}
\Delta \mathcal{E}_{\mu_s}(X_s) \leq \kappa \mu(V) + (\kappa + a) V(X_s) + \frac{\kappa}{r+s}\phi(s).
\end{equation*}
iii) The domination condition (\ref{domination}) leads also to
\begin{equation*}
W(X_s,X_s) \leq 2 \kappa V(X_s).
\end{equation*}
Putting all the pieces together, we find the following inequality
(denoting by $C_1$ a deterministic positive constant) for $t$ large enough
\begin{equation*}
\phi(t) \leq k_{\epsilon}t + \epsilon\left( (a+2\kappa)\phi(t) + \kappa\int_0^t \frac{\mathrm{d}s}{r+s}\phi(s)
+ C_1t + \frac{4 \kappa}{r}\phi(t)\right).
\end{equation*}
Now we choose $\epsilon$ small enough such that $\epsilon \left(a+\kappa(2+\frac{4}{r})\right)<1$ and so, we have:
\begin{equation*}
\phi(t) \leq C'_1t+ \int_0^t \frac{C'_2\mathrm{d}s}{r+s}\phi(s),
\end{equation*}
with $C'_i$ another positive constants. Finally, thanks to Gronwall's lemma, there exists a positive deterministic constant $\beta$ such that $\phi(t) \leq \beta t$ for $t$ large
enough as required (because $C'_2<1$).
\end{proof}

\begin{propo}
Let $\beta>1$ such that $\mu_t\in
\mathcal{P}_\beta(\mathbb{R}^d;V)$ for $t$ large enough. For all
$n\in \mathbb{N}$, we have that $\mathbb{E}_{x,r,\mu}(V^n(X_t))$
is bounded.
\end{propo}
\begin{proof}
We drop the subscripts $x,r,\mu$ in the following. We will prove the result for the Lyapunov function $\mathcal{E}_\mu(x)$ instead of $V$. Let $n=1$. We apply the It\^o formula to $(s,x) \mapsto \mathcal{E}_{\mu_s}(x)$:
\begin{eqnarray*}
\mathrm{d}\mathcal{E}_{\mu_s}(X_s) &=& (\nabla
\mathcal{E}_{\mu_s}(X_s), \mathrm{d}B_s) - \left|\nabla
\mathcal{E}_{\mu_s}(X_s)\right|^2 \mathrm{d}s + \frac{1}{2}\Delta
\mathcal{E}_{\mu_s}(X_s)\mathrm{d}s\\
&+& \left(W(X_s,X_s) - W*\mu_s(X_s)\right)\frac{\mathrm{d}s}{r+s}.
\end{eqnarray*}
The condition (\ref{curvature}) on $V+W*\mu$ (uniform in $\mu$) leads to (for $t$ large enough):
$$\forall \alpha>0, \exists K_\alpha = K(\alpha,\beta,V); \, 
\mathcal{E}_{\mu_t}(X_t) \leq \alpha \vert\nabla
\mathcal{E}_{\mu_t}(X_t) \vert^2 + K_\alpha.$$ From the domination
condition (\ref{domination}) on $W$ and the growth condition on
$V$, there exists $a>0$ such that $\Delta \mathcal{E}_{\mu_t}(X_t)\leq a\mathcal{E}_{\mu_t}(X_t)$.
These bounds lead, for all $t\geq s$ large enough, to
\begin{eqnarray*}
\mathbb{E}\mathcal{E}_{\mu_t}(X_t)&\leq &
\mathbb{E}\mathcal{E}_{\mu_s}(X_s) + \frac{1}{2\alpha}\int_s^t
\left(K_\alpha- \mathbb{E}\mathcal{E}_{\mu_u}(X_u)
\right)\mathrm{d}u + \frac{a}{2}\int_s^t \mathbb{E}\mathcal{E}_{\mu_u}(X_u) \mathrm{d}u \\
&+& \kappa \int_s^t \mathbb{E}V(X_u) (r+u)^{-1} \mathrm{d}u.
\end{eqnarray*}
Now, we can choose $\alpha$ such that $1/\alpha - a = 2a$ and we
recall that $V(X_t) = O(t)$. Therefore the preceding inequality
becomes with $M=M(\beta,V)$
\begin{eqnarray*}
\mathbb{E}\mathcal{E}_{\mu_t}(X_t)\leq
\mathbb{E}\mathcal{E}_{\mu_s}(X_s) -a\int_s^t
\mathbb{E}\mathcal{E}_{\mu_u}(X_u)\mathrm{d}u + M(t-s)
\end{eqnarray*}
We divide both sides by $t-s$ and let $s\rightarrow t$. Let $x(t) := \mathbb{E}\mathcal{E}_{\mu_t}(X_t)$. Solving the preceding inequality boil down to solve $\dot{x}\leq M - ax$. The solution satisfies $x(t)\leq \left(x(0)+M\int_0^t e^{as}\mathrm{d}s \right)e^{-at}$ and we finally obtain:
\begin{eqnarray*}
\mathbb{E}\mathcal{E}_{\mu_t}(X_t)&\leq & K V(x) e^{-at} + \frac{M}{a}(1-e^{-at}).
\end{eqnarray*}
We conclude the general case $n\geq 1$ by induction.
\end{proof}

\subsection{Asymptotic behavior}
We define the family of measures $\{\varepsilon_{t,t+s}; t\geq 0, s\geq 0 \}$ by
\begin{eqnarray}
\varepsilon_{t,t+s} := \int_t^{t+s} (\delta_{X_{h(u)}}-\Pi(\mu_{h(u)}))\mathrm{d}u.
\end{eqnarray}
This family will play an important role: it will be essential for proving that $t\mapsto \mu_{h(t)}$ is an asymptotic pseudotrajectory for $\Phi$.

\begin{propo} \label{theorem} i) Let $t$ be large enough. For all $f\in \mathcal{C}^\infty(\mathbb{R}^d;V)$ and every $T>0$ there exists a positive constant $K = K(V,W,x)$ 
such that for all $\delta>0$ $$\mathbb{P}_{x,r,\mu}\left(\sup_{0\leq s\leq T} \vert \varepsilon_{t,t+s} f\vert >\delta \right) \leq K \delta^{-2} e^{-t} \vert\vert f\vert\vert_V^2 .$$
ii) For all $T>0$ and all $f\in \mathcal{C}^\infty(\mathbb{R}^d;V)$, we have $\mathbb{P}_{x,r,\mu}-$a.s. $$\underset{t\rightarrow\infty}{\lim}\, \sup_{0\leq s\leq T} |\varepsilon_{t,t+s}f| =0.$$
\end{propo}
\begin{proof}
i) We need the uniform estimates on the family of semigroups
$(P_t^\mu)$ proved in Section 5. Let $f\in 
\mathcal{C}^\infty(\mathbb{R}^d;V)$. We begin to rewrite
$$\varepsilon_{t,t+s}f = \int_{h(t)}^{h(t+s)}
A_{\mu_u}Q_{\mu_u}f\frac{\mathrm{d}u}{r+u}.$$ We consider the
$\mathcal{C}^2$-valued process $(t,x) \mapsto Q_{\mu_{h(t)}}
f(x)$, which is of class $\mathcal{C}^2$ and a
$\mathcal{C}^1$-semimartingale. Indeed it is easy to see that
$t\mapsto \mu_{h(t)}$ is a.s. a bounded variation process with
values in $\mathcal{M}(\mathbb{R}^d;V)$. Since Proposition 
\ref{gradQ_mu} shows that $\mu \mapsto Q_\mu f$ is also
$\mathcal{C}^1$, the claim follows by composition. So, we apply
the generalized It\^o formula to $(t,x)\mapsto
h(t)^{-1}Q_{\mu_{h(t)}} f(x)$ and decompose $\varepsilon_{t,t+s}$
in four parts (and we will control each term separately):
$$\varepsilon_{t,t+s}f = \varepsilon_{t,t+s}^{(1)} f +
\varepsilon_{t,t+s}^{(2)}f +\varepsilon_{t,t+s}^{(3)}f +
\varepsilon_{t,t+s}^{(4)}f$$ with
\begin{eqnarray*}
\varepsilon_{t,t+s}^{(1)}f &=& -\frac{1}{h(t+s)}Q_{\mu_{h(t+s)}}f(X_{h(t+s)}) + \frac{1}{h(t)} Q_{\mu_{h(t)}}f(X_{h(t)})\\
\varepsilon_{t,t+s}^{(2)}f &=&-\int_{h(t)}^{h(t+s)} Q_{\mu_u}f(X_u)\frac{\mathrm{d}u}{(r+u)^2} \\
\varepsilon_{t,t+s}^{(3)}f &=& \int_{h(t)}^{h(t+s)} \frac{\partial}{\partial u} Q_{\mu_u}f(X_u) \frac{\mathrm{d}u}{r+u}\\
\varepsilon_{t,t+s}^{(4)}f &=& M^f_{h(t+s)} - M^f_{h(t)}
\end{eqnarray*}
where  $M_t^f$ is the local martingale $M_t^f := \int_0^t \nabla Q_{\mu_u}f(X_u) \frac{\mathrm{d}B_u}{r+u}.$

We recall the estimates of Propositions \ref{Q_muC1} and \ref{gradQ_mu}: $\forall \varepsilon>0, f\in \mathcal{C}^\infty(\mathbb{R}^d;V)$,
\begin{eqnarray*}
\vert Q_{\mu_{h(t)}} f(X_{h(t)})\vert &\leq & \vert \vert f\vert \vert_V(\varepsilon V(X_{h(t)}) + K(\varepsilon))\\
\vert \nabla Q_{\mu_{h(t)}} f(X_{h(t)}) \vert & \leq & \vert \vert f\vert \vert_V (\varepsilon V(X_{h(t)}) + K_1(\varepsilon)).
\end{eqnarray*}
We also remind that $\int_0^t V(X_s) \mathrm{d}s = O(t)$ and $V(X_t) = O(t)$ a.s. Now, we are able to control each part of $\varepsilon_{t,t+s}$ and find for all $\varepsilon>0$ and $t$ large enough:
\begin{eqnarray*}
|\varepsilon_{t,t+s}^{(1)}f| &\leq & h(t)^{-1}(\vert Q_{\mu_{h(t+s)}}f(X_{h(t+s)})\vert + \vert Q_{\mu_{h(t)}}f(X_{h(t)}) \vert ) \\
&\leq & h(t)^{-1}\vert \vert f\vert \vert_V (\varepsilon(V(X_{h(t+s)}) + V(X_{h(t)})) + 2K(\varepsilon))
\end{eqnarray*}
so $ \underset{0\leq s\leq T}{\sup} |\varepsilon_{t,t+s}^{(1)}f| \leq C_1 h(t)^{-1}\|f\|_V$ a.s.; and similarly
\begin{eqnarray*}
|\varepsilon_{t,t+s}^{(2)}f| \leq \int_{h(t)}^{h(t+s)} (\varepsilon V(X_u) + K(\varepsilon)) \frac{ \mathrm{d}u}{(r+u)^2} \|f\|_V \leq \frac{C_2\| f\|_V }{h(t)^2} \int_{h(t)}^{h(t+s)} V(X_u)\mathrm{d}u 
\end{eqnarray*}
so that $\underset{0\leq s\leq T}{\sup} |\varepsilon_{t,t+s}^{(2)}f| \leq C_2 h(t)^{-1} \|f\|_V$ a.s.

\noindent For the third part of $\varepsilon_{t,t+s}$, we will use Markov's inequality and the bound on the differential of $Q_\mu$ given in Corollary \ref{DQmu}:
\begin{eqnarray*}
\mathbb{P}\left(\sup_{0\leq s\leq T}|\varepsilon_{t,t+s}^{(3)}f| \geq \delta\right) 
&\leq& \delta^{-2} \int_{h(t)}^{h(t+T)} \mathbb{E} \vert (D Q_{\mu_u} \cdot \dot{\mu}_u)(f)(X_u) \vert^2 \frac{\mathrm{d}u}{r+u}\\
&\leq& \frac{C}{\delta^2} \vert \vert f\vert \vert_V^2 \int_{h(t)}^{h(t+T)} \mathbb{E}\left(V^6(X_u)\right) \frac{\mathrm{d}u}{(r+u)^3}.
\end{eqnarray*}
Recall, that we have proved that for all $\varepsilon>0$, $n\in \mathbb{N}$ and $t$ large enough, we obtain
$\mathbb{E}[V^n(X_t)] = o(t^\varepsilon)$. Then, there exists some (uniform) constant $C_3$ such that
$$\mathbb{P}\left(\sup_{0\leq s\leq T}|\varepsilon_{t,t+s}^{(3)}f| \geq \delta\right) \leq \frac{C_3}{\delta^2} h(t)^{-1} \| f \|_V^2 .$$ Since the quadratic variation of $M^f_{h(t+s)}-M^f_{h(t)}$ is bounded by the quantity $\vert \vert f\vert \vert_V^2 \int_{h(t)}^{h(t+T)} (\varepsilon V(X_u) + K_1(\varepsilon))^2 \frac{\mathrm{d}u}{(r+u)^2}$, Burkholder-Davis-Gundy's inequality implies directly
\begin{eqnarray}\label{doob}
\mathbb{P}_{x,r,\mu}\left(\sup_{s\in [0,T]}|\varepsilon_{t,t+s}^{(4)}f|\geq \delta \right)\leq
\frac{C_4}{\delta^2} h(t)^{-1} \vert \vert f\vert \vert_V^2.
\end{eqnarray}

ii) Let $T>0$ and $f\in \mathcal{C}^\infty(\mathbb{R}^d;V)$. We just need to prove that $$\underset{t\rightarrow\infty}{\lim}\,\sup_{0\leq s\leq T} |\varepsilon_{t,t+s}^{(4)} f| = \underset{t\rightarrow\infty}{\lim}\,\sup_{0\leq s\leq T} |\varepsilon_{t,t+s}^{(3)} f| = 0.$$ We will use Borel-Cantelli's lemma. First, for all $\varepsilon >0$, we have by Doob's inequality added to Burkholder-Davis-Gundy's inequality that
\begin{eqnarray*}
\mathbb{P}_{x,r,\mu}\left(\sup_{n\leq t < n+1}\sup_{s\in [0,T]}|\varepsilon_{t,t+s}^{(4)}f|\geq \delta \right) \leq \frac{C}{\delta^2} \vert \vert f\vert \vert_V^2 \sup_{n\leq t <n+1} (\varepsilon + h(t)^{-1}).
\end{eqnarray*}
As it is true for all $\varepsilon >0$, we deduce from the preceding inequality that $$\mathbb{P}_{x,r,\mu}\left(\sup_{n\leq t< n+1}\, \sup_{0\leq s\leq T} |\varepsilon_{t,t+s}^{(4)} f|\geq \delta \right)\leq \frac{C}{\delta^2} \|f\|_V^2 h(n)^{-1}.$$ As we know that $\sum_n h(n)^{-1}$ converges, we conclude by Borel-Cantelli's lemma that a.s. $$\lim_{n\rightarrow \infty} \sup_{n\leq t<n+1}\, \sup_{0\leq s\leq T} |\varepsilon_{t,t+s}^{(4)} f| = 0.$$ The same argument for $|\varepsilon_{t,t+s}^{(3)} f|$ permits to conclude.
\end{proof}

\begin{lemma}\label{lem_pta} (\cite{beLR}) If for all $T>0$, all $f\in \mathcal{C}^\infty(\mathbb{R}^d;V)$,
we have $$\underset{t\rightarrow\infty}{\lim}\, \underset{0\leq s
\leq T}{\sup} |\varepsilon_{t,t+s}f| = 0\, a.s.,$$ then the
time-changed process, given by $\mathbb{R}_+ \rightarrow
\mathcal{P}(\mathbb{R}^d;V)$, $t\mapsto \mu_{h(t)}$ is a.s. an
asymptotic pseudotrajectory for $\Phi$ (for the weak topology of
measures).
\end{lemma}
\begin{proof}
We recall that the family $(\mu_t,t\geq 0)$ is a.s. tight and by
Prokhorov's theorem, as we work in a Polish space, it is equivalent
to the relative compactness of $(\mu_t,t \geq 0)$. Let
$\mathcal{C}(\mathbb{R},\mathcal{P}(\mathbb{R}^d;V))$ be the space
of continuous paths $\nu: \mathbb{R}\rightarrow
\mathcal{P}(\mathbb{R}^d;V)$ equipped with the weak topology. Let
$\theta$ be the translation flow $\theta:
\mathcal{C}(\mathbb{R},\mathcal{P}(\mathbb{R}^d;V))\times \mathbb{R} \rightarrow \mathcal{C}(\mathbb{R},\mathcal{P}(\mathbb{R}^d;V)); \theta_t(\nu)(s) = \nu(t+s)$ and $\hat{\Phi}$ be the mapping
$\hat{\Phi}:\mathcal{C}(\mathbb{R},\mathcal{P}(\mathbb{R}^d;V)) \rightarrow \mathcal{C}(\mathbb{R},\mathcal{P}(\mathbb{R}^d;V)); \hat{\Phi}(\nu)(t) = \Phi_t(\nu(0))$. Bena\"im \cite{dea} (theorem
3.2) asserts that a continuous map $\nu : \mathbb{R}_+\rightarrow \mathcal{P}(\mathbb{R}^d;V)$ is an asymptotic pseudotrajectory for the semiflow $\Phi$ if and only if $\nu$ is uniformly continuous
(for the weak topology) and every limit point of $\{\theta_t(\nu);
t\geq 0\}$ is a fixed point for $\hat{\Phi}$. We begin to prove
that $\mu_{h(t)}$ is uniformly continuous for the weak topology.
We have by definition of $\mu_t$ that
\begin{eqnarray*}
\vert \mu_{h(t+s)}f - \mu_{h(t)}f\vert \leq \int_t^{t+s}
\left(\vert \mu_{h(u)} f \vert + \vert f(X_{h(u)})\vert \right) \mathrm{d}u.
\end{eqnarray*}
As $\int_0^t V(X_u) \mathrm{d}u = O(t)$ a.s., this enables us to show that %there exists $M>0$, depending on $\beta$ only, such that 
for all $t$ large enough
\begin{equation}\label{unifcont}
\vert \mu_{h(t+s)}f - \mu_{h(t)}f\vert \leq %M
2\beta s\vert\vert f\vert\vert_V.
\end{equation}
We put these estimates in equation (\ref{distance}) and the
uniform continuity follows.

Let $I_F: C^0(\mathbb{R},\mathcal{P}(\mathbb{R}^d;V))\rightarrow
C^0 (\mathbb{R}, \mathcal{M}(\mathbb{R}^d;V))$ be the mapping
defined by $$I_F(\nu)(t) := \nu(0) + \int_0^t F(\nu(s)) \mathrm{d}s$$ where $F$ is the vector field $F(\mu) = \Pi(\mu) - \mu$. Then, by definition of $\mu_{h(t)}$ $$\theta_t(\mu_{h(\cdot)}) = I_F(\theta_t(\mu_{h(\cdot)})) + \varepsilon_{t, t+\cdot}.$$
Thus, by relative compactness of $(\mu_{h(t)}, t\geq 0)$ and continuity of $I_F$, we find that $\lim_{t\rightarrow \infty} \varepsilon_{t,t+\cdot}=0$ in $C^0 (\mathbb{R}, \mathcal{M}(\mathbb{R}^d;V))$ if and only if every limit point $\eta$ of $(\theta_t(\mu_{h(\cdot)}))$ satisfies $\eta = I_F(\eta)$, that is $\eta = \hat{\Phi}(\eta)$.
\end{proof}

\begin{thmb} $\mathbb{P}_{x,r,\mu}$-a.s., the function $t\mapsto \mu_{h(t)}$ is an asymptotic pseudotrajectory for $\Phi$.
\end{thmb}
\begin{proof} It suffices to combine Proposition \ref{theorem} with Lemma \ref{lem_pta}.
\end{proof}

\subsection{Back to the dynamical system: a global attractor for the semiflow}
We have defined in Section 4 the smooth dynamical system $\Phi$,
with respect to the strong topology. But, in order to study the
asymptotic behavior of $(\mu_t,t\geq 0)$, it will prove
technically easier to work with the weak topology. Indeed, a good
candidate to be an attractor of the semiflow is the $\omega$-limit
set of $(\mu_t)$ for the weak topology
\begin{equation}
\omega(\mu_t, t\geq 0) := \bigcap_{t\geq 0} \overline{\{\mu_s; s\geq t\}}
\end{equation}
which is (a.s.) weakly compact, since it is contained in
$\mathcal{P}_\beta(\mathbb{R}^d;V)$ a.s. Therefore, we will regard
for now on the semiflow $\Phi$ with the weak topology:
\begin{propo}
$\Phi : \mathbb{R}_{+}\times \mathcal{P}(\mathbb{R}^d;V)$ induces
a continuous semiflow with respect to the weak topology.
\end{propo}
\begin{proof}
Since $\mu\mapsto W*\mu(x)$ is readily weakly continuous (see the
domination condition again), we see that $\Pi$ is weakly
continuous. Now, going back to the Picard approximation scheme of
Section 4, it results that $\mu\mapsto\mu^{(n)}_t$ is weakly
continuous for every $n$ and $t$. Passing to the limit, we are
done.
\end{proof}

Now we need to recall a short list of important definitions coming
from the theory of dynamical systems.
\begin{defn}
a) A subset $A$ of $\mathcal{P}(\mathbb{R}^d;V)$ is an
\textrm{attracting set} (respectively \textrm{attractor}) for
$\Phi$ provided:\begin{enumerate}
    \item $A$ is nonempty, compact for the weak topology and positively invariant, (respectively invariant) and
    \item $A$ has a neighborhood $\mathcal{N}\subset \mathcal{P}(\mathbb{R}^d;V)$ such that
    $\mathrm{d}(\Phi_t(\mu),A)\rightarrow 0$ as $t \rightarrow +\infty$ uniformly in $\mu \in\mathcal{N}$.
\end{enumerate}

b) The \textrm{basin of attraction} of an attractor $K\subset A$
for $\Phi |A = (\Phi_t |A)_t$ is the positively invariant open set
(in $A$) comprising all points whose orbits asymptotically are in
$K$. That is $$B(K,\Phi |A) := \{\mu\in A;
\underset{t\rightarrow\infty}{\lim} \mathrm{d}
(\Phi_t(\mu),K)=0\}.$$

c) A \textit{global attracting set} (respectively \textrm{global
attractor}) is an attracting set (respectively attractor) whose
basin is the whole space $\mathcal{P}(\mathbb{R}^d;V)$.

d) Let $A$ be a positively invariant set for $\Phi$. An attractor
for $\Phi |A$ is \textrm{proper} if it is different from $A$.

e) An \textrm{attractor-free set} is a nonempty compact invariant
set $A$ such that $\Phi |A$ has no proper attractor.
\end{defn}
Our aim is now to describe the limit set of $\mu_t$ and find a global attracting set for $\Phi$. The
natural candidate is the limit set $\omega(\mu_t, t\geq 0)$. First, we describe dynamically the limit set of $\mu_t$.
% \begin{theo}\label{rel_comp}(\cite{beH})
% If the limit set of an asymptotic pseudotrajectory is relatively compact, then it is an attractor-free set.
% \end{theo}
%  and \ref{rel_comp} imply the following
\begin{theo}\label{free_set} The limit set of $\{\mu_t, t\geq 0\}$ is
$\mathbb{P}_{x,r, \mu}$-almost surely an attractor-free set of $\Phi$.
\end{theo}
\begin{proof}
It results from Theorem \ref{PTA} and \cite{beH}.
\end{proof}

\begin{coro} $\mathbb{P}_{x,r,\mu}\left(\underset{t\rightarrow + \infty}{\overline{\lim}} |X_t| = + \infty\right) = 1$.
\end{coro}
\begin{proof}
Let $A$ be a open subset of $\mathbb{R}^d$ such that $\gamma (A)>0$. Since the measure $\gamma$ is diffusive, we have that for all $\nu \in \widehat{\mathrm{Im}(\Pi)}\cap \omega(\mu_t, t\geq 0)$, there exist $m, M>0$ (depending on $\beta$ only) such that $m\gamma \leq \nu \leq
M\gamma.$ Now, if we consider a sequence $(\nu_{t_n}, n\geq 0)$ in
$\mathcal{P}(\mathcal{P}(\mathbb{R}^d;V))$, the limits of its
convergent subsequences will belong to $\widehat{\mathrm{Im}(\Pi)}
\cap \omega(\mu_t, t\geq 0)$, because $\omega(\mu_t, t\geq 0)$ is a.s. an
attractor-free set of $\Phi$. Thus, there exists a subsequence
$(\nu_{t_{n_k}})$ such that $\nu_{t_{n_k}}$
converges almost surely to $\nu$ for the weak topology. For any
smooth function $\varphi$ compactly supported, we have that
$$\nu_{t_n} (\varphi) \xrightarrow{w} \nu (\varphi).$$ If we
consider $\varphi$ such that it equals 1 on $A$ and $0$ out of a
set $B$ containing $A$, we find that $\nu (\varphi) \geq
\nu(A)>0$. Thus $$\nu(B) \geq \limsup \nu_t (\varphi) \geq \liminf
\nu_{t}(\varphi) \geq \nu(A) \geq m \gamma(A).$$ So, it implies
that $\int_0^{t_n} \delta_{X_s} (A) \mathrm{d}s$ is asymptotically
equivalent to $t_n m\gamma(A)$, which in turn gives $\int_0^\infty
\delta_{X_s}(A) \mathrm{d}s = \infty$ a.s. Then, for all constant
$K>0$, $\int_0^\infty \delta_{X_s}(\mathbb{R}^d \setminus
\overline{B}_K)\mathrm{d}s = \infty$ a.s., where $\overline{B}_K$
is the closed ball of radius $K$. Finally
$$\mathbb{P}_{x,r,\mu}\left(\bigcap_K \left\{\int_0^\infty
\mathrm{d}s \1_{\{|X_s| \geq K\}} = \infty \right\} \right) = 1.$$ 
\end{proof}
Second, we look at the (nonempty) set $\widehat{\mathrm{Im}(\Pi)}\cap \omega(\mu_t, t\geq 0)$.
\begin{theo}\label{global_attract}
$\widehat{\mathrm{Im}(\Pi)}\cap \omega(\mu_t, t\geq 0)$ is a.s. a global attracting set for $\Phi$.
\end{theo}
\begin{proof}
We begin to notice that $\widehat{\mathrm{Im}(\Pi)}\cap \omega(\mu_t, t\geq 0)$ is a.s. compact for the weak topology and $\widehat{\mathrm{Im}(\Pi)} \cap \omega(\mu_t, t\geq 0)$ is positively invariant by definition. For all $\mu\in \omega(\mu_t, t\geq 0)$, we assert that  $\mathrm{d} (\Phi_t(\mu), \widehat{\mathrm{Im}(\Pi)}\cap  \omega(\mu_t, t\geq 0))$ converges to 0 uniformly in $\mu$. Indeed, recall that $\omega(\mu_t, t\geq 0)$ is an attractor free set for $\Phi$, so for all $s\geq 0$, $\Phi_s(\mu)\in \omega(\mu_t, t\geq 0)$. As we also already know that $\lim \mathrm{d} (\Phi_t(\mu), \widehat{\mathrm{Im}(\Pi)}) = 0,$ uniformly in $\mu$, the assertion is proved. 
\end{proof}

\begin{lemma}
$\omega(\mu_t, t\geq 0)$ is a.s. a subset of $\widehat{\mathrm{Im}(\Pi)}$.
\end{lemma}
\begin{proof}
As $\mu_{h(t)}$ is an asymptotic pseudotrajectory for the
semiflow, which implies that $\omega(\mu_t, t\geq 0)$ is attractor free, we
have by Theorem \ref{global_attract} that $\omega(\mu_t, t\geq 0)$ is the only attractor of $\Phi$ restricted to this set. Therefore, $\widehat{\mathrm{Im}(\Pi)} \cap \omega(\mu_t, t\geq 0) = \omega(\mu_t, t\geq 0)$.
Consequently, $\omega(\mu_t, t\geq 0)\subset
\widehat{\mathrm{Im}(\Pi)}$.
\end{proof}

When $W$ is symmetric, we can give a better description of $\omega(\mu_t, t\geq 0)$. Let begin with the following:
\begin{theo}(Tromba \cite{tromba})
Let $\mathcal{B}$ be a $\mathcal{C}^\infty$ Banach manifold, $F$ a
$\mathcal{C}^\infty$ vector field on $\mathcal{B}$ and
$\mathcal{E}: \mathcal{B}\rightarrow \mathbb{R}$ a
$\mathcal{C}^\infty$ function. Assume that:
\begin{enumerate}
    \item $D\mathcal{E}(\mu) = 0$ if and only if $F(\mu)=0$;
    \item $F^{-1}(0)$ is compact;
    \item for each $\mu\in F^{-1}(0)$, $D\mathcal{E}(\mu)$ is a
    Fredholm operator.
\end{enumerate}
Then $\mathcal{E}(F^{-1}(0))$ has an empty interior.
\end{theo}

\begin{propo}\label{cours_tromba}(\cite{dea}, proposition 6.4)
Let $\Lambda$ be a compact invariant set for a semiflow $\Phi$ on
a metric space $E$. Assume that there exists a continuous function
$\mathcal{V}: E\rightarrow \mathbb{R}$ such that:
\begin{enumerate}
    \item $\mathcal{V}(\Phi_t(x))< \mathcal{V}(x)$ for $x\in E\backslash
    \Lambda$ and $t>0$;
    \item $\mathcal{V}(\Phi_t(x))= \mathcal{V}(x)$ for $x\in\Lambda$ and
    $t>0$.
\end{enumerate}
If $\mathcal{V}$ has an empty interior, then every attractor-free set $A$ for $\Phi$ is contained in $\Lambda$. Furthermore, $\mathcal{V}$ restricted to $A$ is constant.
\end{propo}

\begin{thma}  Suppose that
$W$ is symmetric. Then the limit set  $\omega(\mu_t, t\geq 0)$ is $\mathbb{P}_{x,r,\mu}$-a.s. a compact connected subset of the fixed points of $\Pi$.
\end{thma}
\begin{proof}
We work only with absolutely continuous probability measures. We want to use Proposition \ref{cours_tromba} with the Lyapunov
function $\mathcal{E}$ (the free energy composed with $\Pi$), which satisfies the required condition. Lemma \ref{lem_pi} asserts that the fixed points of $\Pi$ form a nonempty compact subset of $\mathcal{P}(\mathbb{R}^d;V)$. Let $F(\mu) := \Pi(\mu) -\mu$. We already know that $F^{-1}(0)$ is compact for the weak topology. Therefore, we only need to show that $\mathcal{E}(F^{-1}(0))$ has an empty interior. Let $\mu \in F^{-1}(0)$ and prove that $DF(\mu)$ is a Fredholm operator. Let $\nu\in \mathcal{P}_\beta(\mathbb{R}^d;V)$. Thanks to Lemma \ref{majW}, there exists a constant $C(\beta)$ such that
$\|DF(\mu)\cdot \nu\|_V \leq C(\beta) \|\nu\|_V$. So, the set $\{DF(\mu)\cdot \nu; \|\nu\|_V\leq 1\}$ is bounded. For $x,y\in \mathbb{R}^d$, we get
\begin{eqnarray*}
|DF(\mu)\cdot \nu(x) - DF(\mu)\cdot \nu(y)| &\leq & 2|W*\nu(x) \Pi(\mu)(x) - W*\nu(y) \Pi(\mu)(y)|\\
&+& 2\int W*\nu \mathrm{d}\Pi(\nu)|(\Pi(\mu)(x) - \Pi(\mu)(y))|\\
&\leq &  M (|x-y| \|\mu\|_V + |\mu(x) -\mu(y)|\\
&+& |V(x) - V(y)| + \|W(y,\cdot) - W(x,\cdot)\|_V \|\mu\|_V)
\end{eqnarray*}
So, the map $DF(\mu)\cdot \nu$ ($\|\nu\|_V\leq 1$) is equicontinuous and by Ascoli's theorem, we conclude that the preceding set is relatively compact in $\mathcal{C}^0(\mathbb{R}^d;V)$ and thus the operator $DF(\mu)$ is compact. Moreover, this operator is self-adjoint. It follows from the spectral theory of compact self-adjoint operators that $DF$ has at most countably many real eigenvalues; the set of nonzero eigenvalues is either finite or can be ordered as $|\lambda_1|> |\lambda_2|> \ldots>0$ with $\underset{n\rightarrow\infty}{\lim} \lambda_n = 0$. Therefore, we apply the result of Tromba and $\mathcal{E}(F^{-1}(0))$ has an empty interior. We conclude thanks to Proposition \ref{cours_tromba}.
\end{proof}

\section{Illustration in dimension $d=2$}
When $W$ is not symmetric, it can happen that there exists no Lyapunov function and that the limit set $\omega(\mu_t, t\geq 0)$ is a non trivial orbit. Suppose for instance that (for $d=2$) $W(x,y) = (x,Ry)$ where $R$ is a rotation matrix, $V$ is a polynomial $V(x)= V(|x|) := a|x|^4 + b|x|^2 +1$. Note, that the probability measure $\gamma(\mathrm{d}x) = e^{-2V(x)} \mathrm{d}x/Z$ is invariant by rotation. Then, one expects, depending on $R$ and $V$, that either the unique invariant set for the semiflow is $\{\gamma\}$ and so $\mu_t$ converges a.s. to $\gamma$; or $\mu_t$ converges a.s. to a random measure, related to the critical points of the free energy; or $\omega(\mu_t, t\geq 0)$ is a periodic orbit related to $\gamma$. Remark that, equivalently considering $W(x,y)+\frac{1}{2}(b|x|^2 + |y|^2/b)$ or $W$, we satisfy the set of conditions \textbf{(H)}. We denote $p:=(1,0)^T$.

\begin{lemma}\label{egalites}(\cite{beLR}, lemma 4.6)
For all continuous $\varphi: \mathbb{R}\rightarrow \mathbb{R}$, for all $y\in \mathbb{S}^1$ we
have
$$\int_{\mathbb{R}^2} \left[\varphi((x,y))-\varphi((x,p))\right] \gamma(\mathrm{d}x) = \int_{\mathbb{R}^2} \varphi((x,y)) (x-(x,y)y) \gamma(\mathrm{d}x) =0.$$
\end{lemma}
\begin{proof}
For all $y\in \mathbb{S}^1$, there exists $g\in O(2)$ such that $y=gp$. The first equality follows from a change of variable in the integral (because $V(x) = V(|x|)$). After, define $\phi(y) := \int_{\mathbb{R}^2} \varphi((x,y))(x-(x,y)y)
\gamma(\mathrm{d}x)$. We clearly have $(\phi(y),y) =0$ and the rotation-invariance of $\gamma$ implies for the antisymmetry matrix $h$, $\phi(p) = h\phi(p)$. So, $\phi(p)=0$ and thus $\phi(y) = 0$.
\end{proof}

For any probability measure $\mu\in \mathcal{P}(\mathbb{R}^2;V)$, define the mean of $\mu$ by $\bar{\mu}:=
\int_{\mathbb{R}^2} x\mu(\mathrm{d}x)$. Let the probability measure
\begin{equation}\label{bar(Pi)}
\bar{\Pi}(\bar{\mu})(\mathrm{d}x) :=
\frac{e^{-2(x,R\bar{\mu})}}{Z(\bar{\mu})} \gamma(\mathrm{d}x).
\end{equation}
Here, $\bar{\Pi}(\bar{\mu}) = \Pi(\mu)$. If we let $\overline{\Pi}(\mu) := \int_{\mathbb{R}^2} x \bar{\Pi}(\mu) (\mathrm{d}x)$, then $\bar{\Phi}_t(\mu)$ is readily the semiflow corresponding to
\begin{equation}\label{bar(flow)}
\bar{\Phi}_t(\mu) = e^{-t} \bar{\mu} + e^{-t} \int_0^t e^s \overline{\Pi}(\bar{\Phi}_s(\mu)) \mathrm{d}s, \, \bar{\Phi}_0(\mu) = \bar{\mu}.
\end{equation}

\begin{lemma}\label{diff}
Let $m=\rho v$ with $\rho \geq 0$ and $v\in \mathbb{S}^1$. Then we
get $$\int_{\mathbb{R}^2} x\bar{\Pi}(m)(\mathrm{d}x) =
-\frac{1}{2} \frac{\mathrm{d}}{\mathrm{d}\rho} \log{\left(
\int_{\mathbb{R}^2} e^{-2\rho(x,v)} \gamma(\mathrm{d}x)\right)}
Rv.$$
\end{lemma}
\begin{proof}
One just has to differentiate the function $\alpha \mapsto
\log{\left( \int_{\mathbb{R}^2} e^{-2\alpha(x,v)}
\gamma(\mathrm{d}x)\right)}$ and use the second equality of Lemma
\ref{egalites}.
\end{proof}

Let $m = \rho v$ be the solution to the ODE $\dot{m} = \overline{\Pi}(m) - m$, with $\rho = |m|$ and $v \in \mathbb{S}^1$. Then we 
have by Lemma \ref{diff} that $\dot{ v} = 0.$ Moreover, if we
let $\alpha = 2\rho$, then $\alpha$ satisfies the one-dimensional
ODE
\begin{equation}\label{edobis}
\dot{\alpha} = J(\alpha) = -\alpha + 2 \partial_\alpha \log\left(\int_{\mathbb{R}^2}
e^{-\alpha (x,Rp)}\gamma(\mathrm{d}x) \right).
\end{equation}
Let us define some useful functions expressed in polar coordinates:
\begin{eqnarray*} H(\alpha) &:= &
\int_0^\infty \mathrm{d}\rho \gamma(\rho) \int_0^{2\pi}
\mathrm{d}v e^{-\alpha \rho \cos v}\\
\tilde{H}(\alpha) &:= & \int_0^\infty \mathrm{d}\rho \gamma(\rho)
\rho^2 \int_0^{2\pi} \mathrm{d}v \sin^2v e^{-\alpha \rho \cos v}.
\end{eqnarray*}

\subsection{The case $R=-Id$}

Here, $W$ is a symmetric function. Expressing the problem in polar coordinates, we get
$J(\alpha) = -\alpha \left( 1 - 2\frac{\tilde{H}(\alpha)}{H(\alpha)} \right)$.
\begin{propo}\label{bifurcation}
If $\int_0^\infty \rho^2 \gamma(\rho) \mathrm{d}\rho\leq 1$, then
0 is the unique equilibrium of (\ref{edobis}) and 0 is stable. The
basin of attraction of 0 is $\mathbb{R}_+$.

If $\int_0^\infty \rho^2 \gamma(\rho) \mathrm{d}\rho> 1$, then 0
is linearly unstable and there is another stable equilibrium $\alpha_1$, whose basin
of attraction is $\mathbb{R}_+^*$.
\end{propo}
\begin{proof}
We remark that $J$ is $\mathcal{C}^\infty$. A computation yields to
\begin{eqnarray*}
J^{(3)}(\alpha) &=& 2\frac{H^{(4)}(\alpha)}{H(\alpha)}
-8\frac{H^{(3)}(\alpha)}{H(\alpha)} \frac{H'(\alpha)}{H(\alpha)} +
24 \frac{H''(\alpha)}{H(\alpha)}
\left(\frac{H'(\alpha)}{H(\alpha)}\right)^2 - 12
\left(\frac{H'(\alpha)}{H(\alpha)} \right)^4.
\end{eqnarray*}
We wonder for the sign of $J^{(3)}$. This function corresponds to (twice) the kurtosis of the projection on the axis $x$ of a random variable $X$ (expressed in polar coordinates) such that $X$ has the law $\gamma$. As the graph of the symmetric part of the density function cuts exactly twice the graph of the corresponding Gaussian variable (with the same mean and variance), the kurtosis of $X$ is negative, or more exactly $J^{(3)}(\alpha)<0$ for $\alpha>0$ and $J^{(3)}(0)=0$. So, for all $\alpha\geq 0$, we have $J''(\alpha)\leq J''(0) = 0$. Similarly, we find $$J'(\alpha)\leq J'(0) = -1 + \int_0^\infty \mathrm{d}\rho \gamma(\rho)\rho^2.$$ Therefore, if $J'(0)\leq 0$, then $J$ is a decreasing function and as $J(0)=0$, the first result follows. Else $J'(0)>0$. But $J'$ is a non-increasing function and $\underset{\alpha \rightarrow\infty}{\lim} J'(\alpha) = -1$. So, because of the continuity of $J'$, there exists $\alpha_0>0$ such that $J'(\alpha_0)=0$. Moreover, we have $\underset{\alpha \rightarrow \infty}{\lim} J(\alpha) = -\infty$. Finally, there exists a positive solution to $J(\alpha)=0$ if and only if $\int_0^\infty \mathrm{d}\rho \gamma(\rho) \rho^2 > 1$. In that case, the point 0 is unstable and there exists an other equilibrium, which is stable.
\end{proof}

\begin{rk}
The function $t\mapsto \int_0^{2\pi} e^{-t\cos v} \mathrm{d}v$ is the Bessel function $I_0(t)$.
\end{rk}

The next result shows that we can reduce the problem in studying the dynamical system satisfied by $\bar{\mu}$ and then deduce results on $\mu$.
\begin{lemma}\label{free} (\cite{beLR} proposition 3.9, corollary 3.10)
1) Let $L\subset\mathcal{P}_\beta(\mathbb{R}^d;V)$ be an attractor-free set for $\Phi$ and $A \subset \mathcal{P}_\beta(\mathbb{R}^d;V)$ an attractor for $\Phi$. If $L\cap B(A) \neq \emptyset$\footnote{$B(A)$ is the basin of attraction of $A$}, then $L\subset A$.\\
2) Let $(E,d)$ be a metric space, $\bar{\Phi}: E\times \mathbb{R}\rightarrow E$ a semiflow on $E$ and
$G:\mathcal{P}_\beta(\mathbb{R}^d;V)\rightarrow E$ a continuous function. Assume that $G\circ \Phi_t = \bar{\Phi}_t \circ G$. Then almost surely $G(\omega(\mu_t, t\geq 0))$ is an attractor-free set of $\bar{\Phi}$.
\end{lemma}

We can now state and prove the following
\begin{theo}\label{thId}
Consider the self-interacting diffusion on $\mathbb{R}^2$ associated with $W(x,y) = -(x,y)$. Then we have two
different cases:
\begin{enumerate}
    \item If $\int_0^\infty \mathrm{d}\rho \gamma(\rho) \rho^2 \leq 1$, then a.s. $\mu_t \xrightarrow {(w)} \gamma$;
    \item If $\int_0^\infty \mathrm{d}\rho \gamma(\rho) \rho^2 > 1$, then there exists a random variable $v\in \mathbb{S}^1$ such that a.s. $\mu_t\xrightarrow {(w)} \mu_\infty^v$ with $$\mu_\infty^v(\mathrm{d}x) = \frac{e^{\alpha_1(x,v)}}{Z_1}\gamma(\mathrm{d}x),$$ where $Z_1$ is the normalization constant and $\alpha_1$ is the unique positive solution to the equation $J(\alpha) = -\alpha + 2\frac{H'(\alpha)}{H(\alpha)} =0$.
\end{enumerate}
\end{theo} 
\begin{proof}
Let $G: \mathcal{P}_\beta(\mathbb{R}^2;V) \rightarrow
\mathbb{R}^2$ be the mapping defined by $G(\mu) = \bar{\mu}$. By
Lemma \ref{free}, the limit set of $\bar{\mu}_t$ is a.s. an
attractor-free set of $\bar{\Phi}$. When $\int_0^\infty
\mathrm{d}\rho \gamma(\rho) \rho^2 \leq 1$, then 0 is a global
attractor for the dynamical system generated by $\bar{\Phi}$.
Therefore, each attractor-free set of $\bar{\Phi}$ reduces to 0.
So, a.s. $\bar{\mu}_t \xrightarrow {(w)} 0$ and $\omega(\mu_t, t\geq 0) \subset
G^{-1}(0)$. The definitions of $\bar{\Pi}(\bar{\mu})$ and $J$
imply that $G^{-1}(0)$ is invariant under the action of $\Phi$
and, as $\Pi(\Phi_t\big|_{G^{-1}(0)}(\mu))= \gamma$, we have
$$\Phi\big|_{G^{-1}(0)}(\mu) = e^{-t} (\mu-\gamma) + \gamma.$$
Therefore, $\gamma$ is a global attractor for
$\Phi\big|_{G^{-1}(0)}$. Lemma \ref{free} then implies that
each attractor-free set reduces to $\gamma$. By Theorem
\ref{free_set}, we conclude that $\omega(\mu_t, t\geq 0)=\gamma$.

Suppose now that 0 is unstable for $\overline{\Pi} - Id$. For all $f\in \mathcal{C}^\infty(\mathbb{R}^2;V)$, it holds
$$\frac{\mathrm{d}}{\mathrm{d}t}\mu_{h(t)}f = -\mu_{h(t)}f + \Pi(\mu_{h(t)})f + \frac{\mathrm{d}}{\mathrm{d}s} \varepsilon_{t,t+s}\big|_{s=0} f.$$ If we consider the projection map $P_i(x) = x_i$, then $\partial_t \bar{\mu}_{h(t)} =\overline{\Pi}(\bar{\mu}_{h(t)}) - \bar{\mu}_{h(t)}) + \eta_t$ where $\eta_t$ is the random vector $\eta_t = \frac{\mathrm{d}}{\mathrm{d}s} \varepsilon_{t,t+s}\big|_{s=0}(P_1, P_2)^T$. As 0 is an unstable linear equilibrium for $\overline{\Pi} - Id$, we apply the result of Tarr\`es (\cite{pierre}, part 3) to prove that $\mathbb{P}\left(\lim_{t\rightarrow \infty} \bar{\mu}_{h(t)} = 0\right) = 0.$ Thanks to Theorem \ref{PTA}, we obtain that 
$\underset{t\rightarrow \infty}{\lim} \underset{0\leq s\leq T}{\lim}|\bar{\mu}_{h(t+s)} -\bar{\Phi}_s(\bar{\mu}_{h(t)})| = 0$. We remind that
\begin{equation} \label{rho}
\dot{ \rho} = -\rho - \frac{H'(\alpha)}{H(\alpha)}
\end{equation}
and we denote by $\alpha_1$ the unique positive solution to $-\alpha + 2\frac{H'(\alpha)}{H(\alpha)} =0$. We introduce the invariant set (for $\bar{\Phi}$) $A:= \{m=\rho v; \rho = \alpha_1, v\in \mathbb{S}^1\}.$ As the limit set of $\bar{\mu}_{h(t)}$ is an attractor-free set by Lemma \ref{free}, the ODE (\ref{rho}) implies that $\omega(\bar{\mu}_{h(t)})$ either reduces to $\{0\}$, or is included in $A$. But as $\mathbb{P}\left(\underset{t\rightarrow \infty}{\lim} \bar{\mu}_{h(t)} = 0\right) = 0$, the limit set of $\bar{\mu}_{h(t)}$ is a.s. a subset of $A$. Moreover, as $\dot{v} =0$, we have $\bar{\Phi}_t\big|_A = Id\big|_A$. So, $\bar{\mu}_{h(t)}$ is a Cauchy sequence in $A$ and then there exists $v\in \mathbb{S}^1$ such that $$\underset{t\rightarrow \infty}{\lim}|\bar{\mu}_{h(t)}-\alpha_1 v|=0.$$ To conclude, we have on one side that the limit set of ($\mu_t$) is an attractor-free set for $\Phi\big|_{G^{-1}(\alpha_1 v)}$ and on the
other side, that the semiflow $\Phi\big|_{G^{-1}(\alpha_1 v)}$ admits $\mu_{\infty}^v$ as a global attractor. This leads to
$\omega(\mu_t, t\geq 0) = \mu_\infty^v$.
\end{proof}

\subsection{The case ``$R$ is a rotation"}
We assume that $R=R(\theta)$ is defined by $R=\left(\begin{array}{cc} \cos
\theta\, \, \sin\theta\\ -\sin\theta\, \cos\theta \end{array} \right)$, with $0\leq \theta< 2\pi$. We emphasize that (unless $\theta = 0,\pi$) $W$ is not a symmetric function.

\begin{theo}\label{generalth}
Consider the self-interacting diffusion on $\mathbb{R}^2$
associated with $W(x,y) =(x,Ry)$. Then one of the following
holds:
\begin{enumerate}
    \item If $V$ is such that $\int_0^\infty \mathrm{d}\rho \gamma(\rho) \rho^2 \cos (\theta) > -1$,
    then a.s. $\mu_t \xrightarrow {(w)} \gamma$;
    \item If $V$ is such that $\int_0^\infty \mathrm{d}\rho \gamma(\rho) \rho^2 \cos(\theta)\leq -1$, then we get two cases:

    a) if $\theta = \pi$ then there exists a random variable $v\in \mathbb{S}^1$ such that a.s. $\mu_t\xrightarrow {(w)} \mu_\infty^v$ with  $\mu_\infty^v(\mathrm{d}x) = \frac{e^{\alpha_1(x,v)}}{Z_1}\gamma(\mathrm{d}x),$ where $Z_1$ is the normalization constant and $\alpha_1$ is the unique positive solution to $-\alpha + 2\frac{H'(\alpha)}{H(\alpha)} =0$,

    b) if $\theta \neq \pi$, then $\omega(\mu_t, t\geq 0) = \{\nu(\delta), 0\leq \delta<2\pi\}$ a.s., where $\nu(\delta)=\frac{1}{e^{T_\theta}-1}\int_0^{T_\theta} e^s \mu_\infty^{v,\theta} \mathrm{d}s,$ with $T_\theta = 2\pi(\tan \theta)^{-1}$ and $\mu_\infty^{v,\theta}$ is the unique positive solution to $- \alpha + 2\cos\theta \frac{H'(\alpha)}{H(\alpha)} =0$.
\end{enumerate}
\end{theo}
\begin{proof}
Let $v = gp$ with $g\in O(2)$ and $m = \alpha v/2$. We remind the equations
\begin{eqnarray*}
\dot{\alpha} = -\alpha - 2\frac{H'(\alpha)}{H(\alpha)}(Rv,v);\,\,
\dot{v} = -\frac{2}{\alpha} \frac{H'(\alpha)}{H(\alpha)}\left( (Rv,v)v - Rv\right).
\end{eqnarray*}
But, by definition of $R$ and $v= (\cos\sigma,\sin\sigma)^T$, a simple computation yields to
$$(Rv,v)v - Rv = \left(\begin{array}{c}  -\sin\theta\sin\sigma \\ \sin\theta \cos\sigma
\end{array}\right).$$ We finally get after some easy calculations
\begin{eqnarray}\label{dyn2}
\left\{%
\begin{array}{ll}
    \dot{ \alpha} \, =\, -\alpha -\frac{2H'(\alpha)}{H(\alpha)}\cos\theta;\\
    \dot{ \sigma} \, =\, \frac{2H'(\alpha)}{\alpha H(\alpha)} \sin\theta. \\
\end{array}%
\right.
\end{eqnarray}
We recall that $\frac{H'(\alpha)}{H(\alpha)}>0$ for $\alpha>0$. By Proposition \ref{bifurcation}, we have a bifurcation at $\cos \theta \int_0^\infty\gamma(\mathrm{d}\rho)\rho^2 = 1$. More precisely, if $\cos \theta \int_0^\infty \gamma(\mathrm{d}\rho)\rho^2 \geq 1$, then the set $\{(\sigma,\alpha); \alpha=0\}$ is a global attracting set for Equation (\ref{dyn2}) and so a.s. $\mu_t\xrightarrow {(w)} \gamma$. If $\cos \theta \int_0^\infty \gamma(\mathrm{d}\rho)\rho^2 < 1$, then $\{(\sigma,\alpha); \alpha = \alpha_1(\cos\theta)\}$ is a global attracting set. On this set, the dynamics is given by $$\dot{\sigma} = \frac{2H'(\alpha_1(\cos\theta))}{\alpha_1(\cos\theta)    H(\alpha_1(\cos\theta))} \sin\theta = \tan \theta.$$ By Theorem \ref{thId}, we show that there exists a random variable $\sigma_0$ such that a.s.
\begin{equation}\label{egal}
\underset{t\rightarrow \infty}{\lim} 
\left|\bar{\mu}_{h(t)} -\frac{\alpha_1(\cos\theta)}{2} v(t\tan\theta + \sigma_0) \right|=0.
\end{equation}
At that point, we know the dynamics on the set $\tilde{A} := \{(\sigma,\alpha); \alpha =\alpha_1(\cos\theta)\}$. Unfortunately, we need more to conclude: we have to study the coupled system defined on
$\mathcal{M}(\mathbb{R}^2;V)\times \mathbb{R}^2$ by
\begin{equation}\label{dyn3}
\left\{%
\begin{array}{ll}
 \dot{ m} = -m + \bar{\Pi}(m);\\
   \dot{ \nu} = -\nu + \bar{\Pi}(m). \\
\end{array}%
\right.
\end{equation}
By Lemma \ref{free}, $\omega(\mu_t, t\geq 0)\times \tilde{A}$ is an
attractor-free set for the preceding semiflow restricted to $\mathcal{P}(\mathbb{R}^2;V)\times\mathbb{R}^2$. The
dynamics on $\omega(\mu_t, t\geq 0)\times \tilde{A}$ is given by
\begin{equation}\label{dyn4}
\left\{%
\begin{array}{ll}
   \dot{ \sigma} = \tan \theta;\\
   \dot{ \nu} = -\nu + f(\sigma) = -\nu + \mu_\infty^{v,\cos\theta}. \\
\end{array}%
\right.
\end{equation}
As the set $\omega(\mu_t, t\geq 0)\times \tilde{A}$ is compact (for the weak topology) and invariant in $\mathcal{P}(\mathbb{R}^2;V)\times \mathbb{R}^2$, we concude by following the lines of \cite{beLR} (theorem 4.11).
\end{proof}

\medskip

\noindent Aline \textsc{Kurtzmann},
Oxford University, Mathematical Institute, 24-29 St Giles', Oxford, OX1 3LB, United Kingdom.
\textit{kurtzmann@maths.ox.ac.uk}

\begin{thebibliography}
\notag

\bibitem{bak} \textsc{Bakry D.} [1994], L'hypercontractivit\'e et son
utilisation en th\'eorie des semigroupes, \textit{Lectures on Prob.
Th. and Stat., Ecole de Prob. de St-Flour}, Springer, 1-114.

\bibitem{dea} \textsc{Bena\"im M.} [1999], Dynamics of stochastic
approximation algorithms, \textit{S\'em. Prob. XXXIII, Lecture Notes
in Math.} \textbf{1709}, 1-68, Springer.

\bibitem{beH} \textsc{Bena\"im M. \textit{\&} Hirsch M.W.} [1996], Asymptotic pseudotrajectories and chain
reccurent flows, \textit{J. Dynam. Diff. Equ.} \textbf{8},
141-176.

\bibitem{beLR} \textsc{Bena\"im M., Ledoux M. \textit{\&}
Raimond O.} [2002], Self-interacting diffusions, \textit{Prob.
Theory Relat. Fields} \textbf{122}, 1-41.

\bibitem{beR} \textsc{Bena\"im M. \textit{\&}
Raimond O.} [2005], Self-interacting diffusions III: symmetric
interactions, \textit{Ann. Prob.} \textbf{33}(5), 1716-1759.

\bibitem{crLJ} \textsc{Cranston M. \textit{\&} Le Jan Y.} [1995], Self-attracting
diffusions: two cases studies, \textit{Math. Ann.} \textbf{303},
87-93.

\bibitem{crM} \textsc{Cranston M. \textit{\&} Mountford T.S.} [1996],
The strong law of large numbers for a Brownian polymer,
\textit{Ann. Prob.} \textbf{2}(3), 1300-1323.

\bibitem{daS} \textsc{Davies E.B \textit{\&} Simon B.} [1984],
Ultracontractivity and the heat kernel for Schr\"odinger operators
and Dirichlet Laplacians, \textit{J. Func. An.} \textbf{59},
335-395.

\bibitem{duR} \textsc{Durrett R.T. \textit{\&} Rogers L.C.G.} [1992],
Asymptotic behavior of Brownian polymers, \textit{Prob. Th. Rel.
Fields} \textbf{92}(3), 337-349.

\bibitem{heR} \textsc{Herrmann S. \textit{\&} Roynette B.} [2003], Boundedness and convergence of some self-attracting diffusions, \textit{Math. Ann.}
\textbf{325}(1), 81-96.

\bibitem{koM} \textsc{Kontoyiannis I. \textit{\&} Meyn S.P.} [2003], Spectral Theory and Limit Theory for Geometrically Ergodic Markov Processes, \textit{Ann. App. Prob.}
\textbf{13}, 304-362.

\bibitem{kun} \textsc{Kunita H.} [1990], \textit{Stochastic flows and stochastic differential equations}, Cambridge studies in advanced mathematics.

\bibitem{led} \textsc{Ledoux M.} [2000], The geometry of Markov diffusion generators,
\textit{Ann. Fac. Sci. Toulouse} \textbf{IX}, 305-366.

\bibitem{mc} \textsc{McCann R.} [1997], A convexity principle for interacting gases,
\textit{Adv. Math.} \textbf{128}, 153-179.

\bibitem{meT} \textsc{Meyn S.P. \textit{\&} Tweedie R.L.} [1993],
\textit{Markov Chains and Stochastic Stability}, Springer-Verlag.

\bibitem{moT} \textsc{Mountford T.S. \textit{\&} Tarr\`es P.} [2007],
An asymptotic result for Brownian polymers, \textit{Ann. Inst. H.
Poincar\'e, Prob. Stat.}, to appear.

\bibitem{pema} \textsc{Pemantle R.} [2007], A survey of random processes with reinforcement,
\textit{Prob. Surveys} \textbf{4}, 1-79.

\bibitem{rai} \textsc{Raimond O.} [1997], Self-attracting diffusions:
case of constant interaction, \textit{Prob. Theory Relat. Fields}
\textbf{107}, 177-196.

\bibitem{roWa} \textsc{R\"ockner M. \textit{\&} Wang F.Y.} [2003], Supercontractivity and Ultracontractivity for (Nonsymmetric) Diffusions Semigroups on Manifolds,
\textit{Forum Math.} \textbf{15}, 893-921.

\bibitem{rot} \textsc{Rothaus O.} [1981], Diffusions on compact
Riemannian manifolds and logarithmic Sobolev inequalities,
\textit{J. Func. An.} \textbf{42}, 102-109.

\bibitem{pierre} \textsc{Tarr\`es P.} [2000], Pi\`eges r\'epulsifs, \textit{C. R. Acad. Sci. Paris S\'er. I Math.} \textbf{330}, 125--130.

\bibitem{tromba} \textsc{Tromba A.J.} [1977], The Morse-Sard-Brown theorem for functionals and the problem of Plateau, \textit{Amer. J. Math.} \textbf{99}, 1251-1256.

\bibitem{cedric} \textsc{Villani C.} [2003], \textit{Topics in Optimal Transportation},
Graduate Studies in Mathematics, Vol. 58, \textit{AMS}.
\end{thebibliography}
\end{document}